\documentclass[a4paper]{amsart}
\usepackage[utf8]{inputenc}
\usepackage[english]{babel}
\usepackage{amsmath, amsfonts, amssymb, amsthm, mathtools}
\usepackage{enumitem}
\usepackage[T1]{fontenc}
\usepackage{lmodern}
\usepackage{microtype}

\usepackage[lite]{amsrefs}

\renewcommand{\PrintDOI}[1]{\href{http://dx.doi.org/\detokenize{#1}}{doi: \detokenize{#1}}}

\BibSpec{book}{%
  +{}  {\PrintPrimary}                {transition}
  +{,} { \textit}                     {title}
  +{.} { }                            {part}
  +{:} { \textit}                     {subtitle}
  +{,} { \PrintEdition}               {edition}
  +{}  { \PrintEditorsB}              {editor}
  +{,} { \PrintTranslatorsC}          {translator}
  +{,} { \PrintContributions}         {contribution}
  +{,} { }                            {series}
  +{,} { \voltext}                    {volume}
  +{,} { }                            {publisher}
  +{,} { }                            {organization}
  +{,} { }                            {address}
  +{,} { \PrintDateB}                 {date}
  +{,} { }                            {status}
  +{}  { \parenthesize}               {language}
  +{}  { \PrintTranslation}           {translation}
  +{;} { \PrintReprint}               {reprint}
  +{.} { }                            {note}
  +{.} {}                             {transition}
  +{,} { \PrintDOI}                   {doi}
  +{} { available at \url}            {eprint}
  +{}  {\SentenceSpace \PrintReviews} {review}
}

\setlist[enumerate,1]{label=\textup{(\alph*)}}% ensure enumerates in theorems are upright letters

\usepackage[arrow,matrix,curve]{xy}
\usepackage[colorlinks,linkcolor=red,citecolor=blue,urlcolor=blue,hypertexnames=true]{hyperref}

\newtheorem{theorem}{Theorem}[section]
\newtheorem{lemma}[theorem]{Lemma}
\newtheorem{proposition}[theorem]{Proposition}
\newtheorem{corollary}[theorem]{Corollary}

\theoremstyle{definition}
\newtheorem{definition}[theorem]{Definition}

\newtheorem{remark}[theorem]{Remark}

\newcommand{\N}{\mathbb{N}}
\newcommand{\Z}{\mathbb{Z}}

\newcommand{\R}{\mathbb{R}}
\newcommand{\C}{\mathbb{C}}
\newcommand{\K}{\mathbb{K}}		% for compact operators

\DeclareMathOperator{\Ext}{Ext}    % extension groups
\DeclareMathOperator{\coker}{coker}% cokernel of a homomorphism

\newcommand{\cG}{\mathcal{G}} % for the crossed module
\newcommand{\cH}{\mathcal{H}} % for its strictification
\newcommand{\Br}[2]{\mathrm{Br}_{#1}(#2)} % Brauer group
\newcommand{\Pic}[2]{\mathrm{Pic}_{#1}(#2)} % Picard group

\newcommand{\Forg}{\mathrm{For}} % forgetful map

\newcommand{\op}{\mathrm{op}} % opposite algebra
\newcommand{\pr}{\mathrm{pr}} % projection

\newcommand{\torus}{\mathbb{T}}

\newcommand{\Aut}[1]{\mathrm{Aut}(#1)}

\newcommand{\ima}{\mathrm{i}}% imaginary unit

\newcommand{\Ad}[1]{\mathrm{Ad}_{#1}}

\newcommand{\id}[1]{\mathrm{id}_{#1}}
\newcommand{\cat}[1][]{\mathcal{C}_{#1}}

\newcommand{\extp}[2]{\Lambda^{#1}#2}% exterior power
\newcommand{\extpd}[2]{\Omega^{#1}#2}% dual of exterior power

\newcommand*{\prto}{\twoheadrightarrow}

\title[Crossed module actions on continuous trace $\mathrm{C}^*$-algebras]{Crossed module actions\\ on continuous trace $\mathrm{C}^*$-algebras}
\author{Ralf Meyer}
\email{rmeyer2@uni-goettingen.de}
\address{Mathematisches Institut\\
 Georg-August-Universit\"at G\"ottingen\\
 Bunsenstra\ss e 3--5\\
 37073 G\"ottingen\\
 Germany}

\author{Ulrich Pennig}
\email{u.pennig@uni-muenster.de}
\address{Mathematisches Institut\\Westf\"alische Wilhelms-Universit\"at M\"unster\\Ein\-stein\-stra\ss e 62\\48149 M\"unster\\Germany}

\begin{document}

\begin{abstract}
  We lift an action of a torus~\(\torus^n\) on the spectrum of a
  continuous trace algebra to an action of a certain crossed module
  of Lie groups that is an extension of~\(\R^n\).  We compute
  equivariant Brauer and Picard groups for this crossed module and
  describe the obstruction to the existence of an action of~\(\R^n\)
  in our framework.
\end{abstract}

\maketitle

\section{Introduction}
% referee suggestions:
% MR2080959 T-duality: topology change from H-flux (added)
% MR2116165 Topology and H-flux of T-dual manifolds (added)
% Bouwknegt-Hannabuss-Mathai:Cstar_in_tensor C*-algebras in tensor categories
% (added)
% MR2989461 Topology and flux of T-dual manifolds with circle actions (added)
% MR2246781 T-duality for non-free circle actions (added)

% Hannabuss-Mathai:Nonassociative_deformation
%   Nonassociative strict deformation quantization of C*-algebras
% MR2985336 Topological T-duality for torus bundles with monodromy
% MR3282986 Topological T\nobreakdash-duality for general circle bundles
% MR3285613 T-duality for circle bundles via noncommutative geometry

The low energy approximation of string theory is described by a
spacetime~$M$ together with a $B$-field over~$M$, whose field strength
gives rise to a class $\delta \in H^3(M,\Z)$. To such a pair, string
theory associates a conformal field theory. It was observed
in~\cite{Buscher:Path} that certain transformations, now known as
Buscher rules, preserve the theory.  While the Buscher rules only
apply in local charts, a global topological description for principal
circle bundles was found
in~\cite{Bouwknegt-Evslin-Mathai:T-duality_H-flux}.  This
generalization of the Buscher rules is now called topological
T-duality.  Its effect on Chern and Dixmier--Douady classes was worked
out in~\cite{Bouwknegt-Evslin-Mathai:Topology_H-flux}.

Topological T-duality was generalized to the case where $M = P \to X$
is a principal torus bundle over a base space $X$. This setup has a very
elegant description in terms of non-commutative geometry
\cites{Mathai-Rosenberg:T-duality, Mathai-Rosenberg:T-duality_II}:
Let~\(P\) be a principal \(\torus^n\)-bundle over~\(X\) and
let~\(A\) be a continuous trace \(\mathrm{C}^*\)-algebra with spectrum~\(P\).
The action of~\(\torus^n\) does not always lift to a
\(\torus^n\)-action on~\(A\): this may fail already for \(n=1\).
Lifting the action to an action of~\(\R^n\) on~\(A\) works much more
often: for this, we only need to know that~\(A\) restricts to
trivial bundles over the orbits of the action
(see~\cite{Crocker-Kumjian-Raeburn-Williams:Brauer}).  If an
\(\R^n\)-action on~\(A\) exists and if \(\hat{A} = A\rtimes\R^n\) is
again a continuous trace \(\mathrm{C}^*\)-algebra with spectrum~\(\hat{P}\),
then \(\hat{A}\) is considered T-dual to~\(A\), and the Connes--Thom
isomorphism \(\textup{K}_*(A)\cong \textup{K}_{*+n}(\hat{A})\) is the
T\nobreakdash-duality
isomorphism between the twisted K-theory groups of \(P\)
and~\(\hat{P}\) given by \(A\) and~\(\hat{A}\), respectively, see
\cites{Mathai-Rosenberg:T-duality, Mathai-Rosenberg:T-duality_II}.
The \(\mathrm{C}^*\)-algebra~\(\hat{A}\) may fail to have continuous trace,
for instance, by being a principal bundle of noncommutative tori.
In this case, one may still consider the \(\mathrm{C}^*\)-algebra~\(\hat{A}\)
to be a T-dual for~\(A\).

If there is no action of~\(\R^n\) on~\(A\),
then~\cite{Bouwknegt-Hannabuss-Mathai:Nonassociative_tori} suggests
non-associative algebras that may play the role of the T-dual.
Another situation where non-associative algebras appear naturally
are Fell bundles over crossed modules or, equivalently, actions of
crossed modules on \(\mathrm{C}^*\)-algebras (see
\cites{Buss-Meyer-Zhu:Non-Hausdorff_symmetries,
  Buss-Meyer:Crossed_products}).  We relate these two appearances of
non-associativity: the action of~\(\torus^n\) always lifts to an
action of a certain crossed module on~\(A\).  Its actions are
equivalent to a certain type of non-associative Fell bundle
over~\(\R^n\).  Such a Fell bundle may be viewed as a continuous
spectral decomposition for the dual \(\R^n\)-action on a
non-associative algebra of the type studied
in~\cite{Bouwknegt-Hannabuss-Mathai:Nonassociative_tori}.  Thus
crossed module actions give an alternative framework for understanding the
non-associative algebras proposed
in~\cite{Bouwknegt-Hannabuss-Mathai:Nonassociative_tori}.

T\nobreakdash-duality is most often formulated for principal circle
bundles.  The case of non-free circle actions with finite stabilizers
has been studied by Bunke and Schick
in~\cite{Bunke-Schick:T-duality_non-free}.  It corresponds to the
study of $\textup{U}(1)$-bundles over orbispaces, which are modeled by
topological stacks in~\cite{Bunke-Schick:T-duality_non-free}.  The
authors also find a sufficient condition for the T\nobreakdash-duality
transformation for a generalized cohomology theory to be an
isomorphism.  T\nobreakdash-duality is generalized to arbitrary circle
actions in~\cite{Mathai-Wu:Topology_flux}, using the Borel
construction.  We treat general $\torus^n$\nobreakdash-actions right
away, that is, we allow~\(P\) to be an arbitrary second countable,
locally compact \(\torus^n\)\nobreakdash-space.  Any
\(\torus^n\)\nobreakdash-action lifts to an action of our crossed
module.

A crossed module of topological groups $\cH = (H^1, H^2, \partial)$
is given by topological groups $H^1$ and~$H^2$ with continuous group
homomorphisms $\partial \colon H^2 \to H^1$ and $c \colon H^1 \to
\Aut{H^2}$ with two conditions that mimic the properties of a normal
subgroup and the conjugation action on the subgroup.  We shall
mostly use the following crossed module:
\begin{equation}
  \label{eq:def_cH}
  \begin{alignedat}{2}
    H^1 &= \R^n \oplus \extp{2}{\R^n},&\quad
    (t_1, \eta_1) \cdot (t_2,\eta_2)
    &= (t_1 + t_2, \eta_1 + \eta_2 + t_1 \wedge t_2),\\
    H^2 &= \extp{2}{\R^n} \oplus \extp{3}{\R^n},&\quad
    (\theta_1,\xi_1) \cdot (\theta_2, \xi_2)
    &= (\theta_1 + \theta_2, \xi_1 + \xi_2),\\
    \partial &\colon H^2 \to H^1,&\quad
    \partial(\theta,\xi) &= (0,\theta),\\
    c&\colon H^1\to\Aut{H^2},&\quad
    c_{(t, \eta)}(\theta, \xi) &= (\theta, \xi + t \wedge \theta),
  \end{alignedat}
\end{equation}
where \(t,t_1,t_2\in\R^n\),
\(\eta,\eta_1,\eta_2,\theta,\theta_1,\theta_2\in\extp{2}{\R^n}\) and
\(\xi,\xi_1,\xi_2\in\extp{3}{\R^n}\).  Here \(\extp{k}{\R^n}\)
denotes the \(k\)th exterior power of~\(\R^n\).  An action of this
crossed module on a \(\mathrm{C}^*\)-algebra~\(A\) consists of two continuous
group homomorphisms \(\alpha\colon H^1\to \Aut{A}\) and \(u\colon
H^2\to U(M(A))\) with \(\alpha_{\partial(h)} = \Ad{u(h)}\) for all
\(h\in H^2\) and \(\alpha_{h}(u_k) = u_{c_h(k)}\) for all \(h\in
H^1\), \(k\in H^2\); in addition, we want~\(\alpha\) to lift the
given action of~\(\R^n\) on the spectrum of~\(A\).

An action \((\alpha,u)\) as above is determined uniquely by
\(\alpha_t = \alpha_{(t,0)}\) for \(t\in\R^n\) and \(u_{t_1\wedge
  t_2} = u_{(t_1\wedge t_2,0)}\) for \(t_1,t_2\in\R^n\) because
\(\alpha_{t_1}\alpha_{t_2}\alpha_{t_1+t_2}^{-1} =
\alpha_{(0,t_1\wedge t_2)}\) and \(\alpha_{t_1}(u_{t_2\wedge t_3}) =
u_{(0,t_1\wedge t_2\wedge t_3)}\).  The unitaries \(u_{t_1\wedge
  t_2}\) are such that \(\Ad{u_{t_1\wedge t_2}} =
\alpha_{(0,t_1\wedge t_2)}\).  So \(t\mapsto \alpha_t\) is an action
of~\(\R^n\) up to inner automorphisms given by~\(u_{t_1\wedge
  t_2}\).  These unitaries are, however, not
\(\alpha_t\)-invariant.  We only know that the unitaries
\(\alpha_{t_1}(u_{t_2\wedge t_3}) = u_{(0,t_1\wedge t_2\wedge
  t_3)}\) are central.

If we divide out the \(\Lambda^3\)-part in~\(H^2\), then we get a
crossed module that is equivalent to~\(\R^n\); its actions are Green
twisted actions of~\(\R^n\), so they can be turned into ordinary
actions of~\(\R^n\) on a \(\mathrm{C}^*\)-stabilisation.  The
assumptions above imply
that~\(\extp{3}{\R^n}\) acts through a map to central,
\(\R^n\)-invariant unitaries in~\(A\), that is, by a map to
\(C(P/\R^n,\torus)\).  A homomorphism \(\extp{3}{\R^n} \to
C(P/\R^n,\torus)\) lifts uniquely to \(\extp{3}{\R^n} \to
C(P/\R^n,\R)\), and such a homomorphism appears as the obstruction
to finding an action of~\(\R^n\) on~\(A\) that lifts the given
\(\torus^n\)-action on~\(P\).  Actions of the crossed module~\(\cH\)
may contain such a lifting obstruction, so that there is no longer
any obstruction to lifting the action to one of~\(\cH\).

Besides proving the existence of liftings, we also classify them up
to equivalence; that is, we compute the equivariant Brauer group
with respect to the crossed module~\(\cH\): There is an exact
sequence of Abelian groups
\[
H^2(P, \Z) \to C(P/\R^n, \extpd{2}{\R^n}) \to \Br{\cH}{P} \prto \Br{}{P}.
\]
The surjection \(\Br{\cH}{P} \prto \Br{}{P}\) says that any
continuous trace \(\mathrm{C}^*\)-algebra over~\(P\) carries an action
of~\(\cH\) lifting the given action of~\(\torus^n\) on the spectrum.
The description of the kernel is the same as for the
\(\R^n\)-equivariant Brauer group, so our result says that whenever
we may lift the action on~\(P\) to an \(\R^n\)-action on~\(A\), then
there is a bijection between actions of \(\cH\) and~\(\R^n\)
on~\(A\).

We also compute the analogue of the equivariant Picard group for our
crossed module actions, and get the same result as in the case of
\(\R^n\)-actions.  Thus the only effect of replacing~\(\R^n\) by the
crossed module~\(\cH\) is to remove the obstruction to the existence
of actions.

Our proof uses a smaller weak \(2\)-group~\(\cG\) that is equivalent
to~\(\cH\).  It consists of the Abelian groups
\(G^1=\coker(\partial_{\cH}) \cong \R^n\) and
\(G^2=\ker(\partial_{\cH}) \cong \extp{3}{\R^n}\), which are linked
by the non-trivial associator
\[
a(t_1,t_2,t_3) = -t_1\wedge t_2\wedge t_3
\qquad\text{for }t_1,t_2,t_3\in\R^n.
\]
We describe morphisms that give an equivalence between \(\cG\) and the
weak \(2\)-group associated to~\(\cH\).  The crossed module actions
above are strict, but there is a more flexible notion of weak action
that makes sense for weak \(2\)-groups and \(2\)-groupoids as well,
see~\cite{Buss-Meyer-Zhu:Higher_twisted}.  The Packer--Raeburn
Stabilisation Trick extends to crossed modules and shows that any weak
action of a crossed module is equivalent to a strict action.  Since
\(\cG\) and~\(\cH\) are equivalent, they have equivalent weak actions
on \(\mathrm{C}^*\)-algebras.  Thus we get the desired strict action of~\(\cH\)
once we construct a weak action of~\(\cG\).

The weak actions we have in mind are equivalent to saturated Fell
bundles in the case of a group action.  For actions of a weak
\(2\)\nobreakdash-group
such as~\(\cG\), they are non-associative Fell bundles over
\(G^1=\R^n\) where the associator is given by unitaries coming from
the action of \(G^2=\extp{3}{\R^n}\).  Allowing weak actions
simplifies the study of equivariant Brauer groups, already in the
group case.  We reprove the obstruction theory for \(\R^n\)-actions
on continuous trace \(\mathrm{C}^*\)-algebras along the way.  The bigroup
version of the result is only notationally more difficult.

The crossed module~\(\cH\) described above can be made slightly
smaller: we may divide out the lattice~\(\extp{3}{\Z^n}\)
inside~\(\extp{3}{\R^n}\), resulting in a compact group.  This is so
because the lifting obstructions that appear for
\(\torus^n\)-actions on continuous trace algebras always vanish on
the lattice~\(\extp{3}{\Z^n}\), and hence so do all the actions
of~\(\cH\) that appear.  This feature, however, is special to
actions of~\(\R^n\) that factor through the standard
torus~\(\R^n/\Z^n\).  The existence result for actions of~\(\cH\)
still works for actions of~\(\R^n\) if the stabiliser lattices are
allowed to vary continuously for different orbits.

To compute the lifting obstruction of a given continuous trace
\(\mathrm{C}^*\)-algebra, it suffices to consider a single free
\(\torus^n\)-orbit, that is, the case of the standard translation
action of~\(\torus^n\) on itself.  Since~\(\R^n\) acts transitively
on~\(\torus^n\), the transformation crossed module \(\cH\ltimes
\torus^n\) is equivalent in a suitable sense to the ``stabiliser''
of a point in~\(\torus^n\); this gives the subcrossed
module~\(\tilde{\cH}\) of~\(\cH\) with
\[
\tilde{H}^1=\Z^n\times \extp{2}{\R^n},\qquad
\tilde{H}^2 = H^2.
\]
We do not develop the full theory of induction for crossed modules
because the relevant special case is easy to do by hand.  It turns out
that actions of~\(\cH\) on continuous trace \(\mathrm{C}^*\)-algebras
over~\(\torus^n\) are equivalent to actions of~\(\tilde{\cH}\) on
continuous trace \(\mathrm{C}^*\)-algebras over the point.  Since our
whole theory is up to equivalence, this is the same as weak actions
of~\(\tilde{\cH}\) on the complex numbers~\(\C\).
Replacing~\(\tilde{\cH}\) by the corresponding sub-\(2\)-group
\(\tilde{\cG}\subseteq \cG\), it is straightforward to classify these.
This also gives the equivariant Brauer group for a single
\(\torus^n\)-orbit, and then allows to identify the lifting
obstruction, up to a sign maybe, with the family of Dixmier--Douady
invariants of the restrictions of~\(A\) to the orbits of the
\(\R^n\)-action.

\subsection*{Acknowledgements}
Part of this paper was finished during the trimester ``Noncommutative geometry and its applications'' at the Hausdorff Institute for Mathematics in Bonn. The authors would like to thank the organisers of this workshop and the staff at the HIM. The second author was supported by the SFB~878 ``Groups, Geometry and Actions.''

\section{A crossed module extension of \texorpdfstring{$\R^n$}{Rn}}
\label{sec:crossed_module_extension}

We construct a crossed module extension~$\cH$ of~$\R^n$ that
circumvents the obstruction to lifting \(\R^n\)-actions from spaces to
continuous trace \(\mathrm{C}^*\)-algebras described
in~\cite{Crocker-Kumjian-Raeburn-Williams:Brauer}.  We also describe a
smaller weak topological \(2\)-group~\(\cG\) equivalent to~\(\cH\).

\begin{definition}
  \label{def:crossed_module}
  A \emph{crossed module} of topological groups $\cH = (H^1,
  H^2, \partial)$ is given by topological groups $H^1$ and~$H^2$
  with continuous group homomorphisms $\partial \colon H^2 \to H^1$
  and $c \colon H^1 \to \Aut{H^2}$, such that
  \[
  \partial(c_{h_1}(h_2)) = h_1\,\partial(h_2)\,h_1^{-1},\qquad
  c_{\partial(h_2)}(h_2') = h_2\,h_2'\,h_2^{-1}
  \]
  for all \(h_1\in H^1\), \(h_2,h_2'\in H^2\); continuity of~\(c\)
  means that the map \(H^1\times H^2\to H^2\), \((h_1,h_2)\mapsto
  c_{h_1}(h_2)\), is continuous.
\end{definition}

A weak \(2\)\nobreakdash-group is a bicategory with only one object
and with invertible arrows and \(2\)\nobreakdash-arrows.  A
\(2\)\nobreakdash-arrow or \(2\)\nobreakdash-cell in a bicategory is
an ``arrow'' between two arrows.  We call \(2\)\nobreakdash-arrows
``bigons'' because they should be drawn like this:
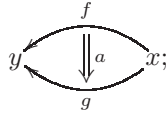
\[
\xymatrix@1@C+2em{
  y &
  \ar@/_1pc/[l]_{f}_{}="0"
  \ar@/^1pc/[l]^{g}^{}="1"
  \ar@{=>}"0";"1"^{a}
  x;
}
\]
here \(x\) and~\(y\) are objects, \(f\) and~\(g\) are arrows \(x\to
y\), and \(a\) is a bigon from~\(f\) to~\(g\).  The standard
reference for bicategories is~\cite{Benabou:Bicategories}; we mainly
follow~\cite{Buss-Meyer-Zhu:Higher_twisted}, which is concerned with
their \(\mathrm{C}^*\)-algebraic applications.

Crossed modules model strict \(2\)-groups, that is,
\(2\)\nobreakdash-groups with trivial unitors and associator.  Let
$\cH = (H^1, H^2, \partial)$ be a crossed module.  The associated
strict $2$-group~$\cat[\cH]$ has a single object and~$H^1$ as its
space of arrows.  The space of bigons is $H^2 \times H^1$ with
source map $(h_2,h_1) \mapsto h_1$ and target map $(h_2,h_1)
\mapsto \partial(h_2)h_1$.  We denote this by $h_2 \colon h_1 \Rightarrow
\partial(h_2)h_1$ or
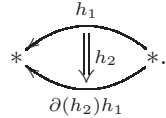
\[
\xymatrix@1@C+2em{
  \ast &
  \ar@/_1pc/[l]_{h_1}_{}="0"
  \ar@/^1pc/[l]^{\partial(h_2)h_1}^{}="1"
  \ar@{=>}"0";"1"^{h_2}
  \ast.
}
\]
The composition of arrows is the group multiplication in~$H^1$.  The
vertical composition of bigons is the multiplication in~$H^2$, and
the horizontal composite of $h_2 \colon h_1
\Rightarrow \partial(h_2)h_1$ and $k_2 \colon k_1
\Rightarrow \partial(k_2)k_1$ is $h_2 \bullet k_2= h_2
c_{h_1}(k_2) \colon h_1 k_1
\Rightarrow \partial(h_2)h_1\partial(k_2)k_1 = \partial(h_2
c_{h_1}(k_2)) h_1 k_1$.

We shall mostly use the crossed module~\(\cH\) defined
in~\eqref{eq:def_cH}.  It is routine to check that~\(\cH\) is a
crossed module of Lie groups.

Since the arrow~\((t,\eta)\) in the Lie \(2\)-group~\(\cat[\cH]\) is
equivalent to~\((t,0)\) by a bigon, we should be able to
shrink~\(\cat[\cH]\), replacing the space of arrows by~\(\R^n\).  This
should reduce the space of bigons to \(\R^n\times \extp{3}{\R^n}\)
because bigons \((t_1,0)\Rightarrow(t_2,0)\) in~\(\cat[\cH]\) are of
the form \((0,\xi)\).  The result of this shrinking is a Lie
\(2\)-group~\(\cG\) equivalent to~\(\cat[\cH]\).  This is no longer
strict, however, so it does not come directly from a
crossed module.  The starting point is the smooth map~\(\Phi\) sending
an arrow~\((t,\eta)\) to the bigon
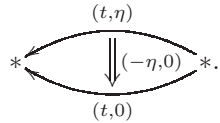
\[
\xymatrix@1@C+4em{
  \ast &
  \ar@/_1pc/[l]_{(t,\eta)}_{}="0"
  \ar@/^1pc/[l]^{(t,0)}^{}="1"
  \ar@{=>}"0";"1"^{(-\eta,0)}
  \ast.
}
\]
This generates a morphism \(F=\Ad{\Phi}\colon \cat[\cH]\to\cat[\cH]\)
of Lie \(2\)\nobreakdash-groups that is equivalent to the identity
functor (morphisms between weak \(2\)-categories are described in
\cite{Buss-Meyer-Zhu:Higher_twisted}*{Definition 4.1}).  We now
describe~\(F\).  It maps \((t,\eta)\in H^1\) to the range \((t,0)\in
H^1\) of the bigon \(\Phi(t,\eta)\).  It maps a bigon
\((\theta,\xi)\colon (t,\eta)\Rightarrow (t,\eta+\theta)\) to the
vertical composite
\[
(t,0) \xLeftarrow{\Phi(t,\eta)}
(t,\eta) \xRightarrow{(\theta,\xi)}
(t,\eta+\theta) \xRightarrow{\Phi(t,\eta+\theta)}
(t,0),
\]
which gives \((0,\xi)\colon (t,0)\Rightarrow(t,0)\).  Since
\(\Phi(t,0)=(0,0)\), our morphism is strictly unital, so the bigon
\(u_*\colon 1_* \Rightarrow F(1_*)\) in the definition of a morphism
is trivial.  A morphism~\(F\) also needs natural bigons
\[
\omega_F((t_1,\eta_1),(t_2,\eta_2))\colon
F(t_1,\eta_1) \cdot F(t_2,\eta_2) \Rightarrow
F((t_1,\eta_1) \cdot (t_2,\eta_2))
\]
describing its compatibility with the multiplication.  We get
\[
\omega_F((t_1,\eta_1),(t_2,\eta_2))\colon
(t_1,0)\cdot (t_2,0) = (t_1+t_2, t_1\wedge t_2) \Rightarrow
(t_1+t_2,0)
\]
by vertically composing the following \(2\)-arrows:
\begin{multline*}
  (t_1,0)\cdot (t_2,0)
  \xRightarrow{\Phi(t_1,\eta_1)^{-1}\bullet \Phi(t_2,\eta_2)^{-1}}
  (t_1,\eta_1)\cdot (t_2,\eta_2)
  \\= (t_1+t_2,\eta_1+\eta_2+t_1\wedge t_2)
  \xRightarrow{\Phi(t_1+t_2,\eta_1+\eta_2+t_1\wedge t_2)}
  (t_1+t_2,0);
\end{multline*}
this gives
\(c_{(t_1,0)}(\eta_2,0) + (\eta_1,0)
- (\eta_1+\eta_2+t_1\wedge t_2,0)
= (-t_1\wedge t_2,t_1\wedge \eta_2)\),
so
\[
\omega_F((t_1,\eta_1),(t_2,\eta_2))
= (-t_1\wedge t_2,t_1\wedge \eta_2).
\]

By construction, \(\Phi\) is a transformation from the identity
functor to the functor~\(F\).  Since all bigons in~\(\cat[\cH]\) are
invertible, this is even an equivalence.  We are going to describe the
range of~\(F\), which is a weak Lie \(2\)-group~\(\cG\).  Its arrows
and bigons are
\begin{align*}
  G^1 &= \R^n
  = \{(t,0)\mid t\in \R^n\} \subseteq H^1,\\
  G^2 &= \R^n\times\extp{3}{\R^n}
  = \{((t,0),(0,\xi))\mid t\in \R^n,\xi\in\extp{3}{\R^n}\}
  \subseteq H^1\times H^2;
\end{align*}
so \((t,\xi)\in \R^n\times\extp{3}{\R^n}\) corresponds to the bigon
\((0,\xi)\colon (t,0)\Rightarrow \partial(0,\xi) (t,0) = (t,0)\) and
thus has range and source~\(t\).  The vertical composition of bigons
adds the \(\xi\)-components as in~\(\cH\).  The composition of arrows
gives
\[
(t_1,0)\cdot_{\cG} (t_2,0)
= F((t_1,0)\cdot (t_2,0))
= F(t_1+t_2,t_1\wedge t_2)
= (t_1+t_2,0),
\]
so it is simply the addition in~\(\R^n\).  The horizontal composition
of bigons in~\(\cG\) is defined by applying~\(F\) to the horizontal
composition in~\(\cat[\cH]\).  Since \(c_{(t,0)}(0,\xi)=(0,\xi)\) for
all \(t\in\R^n\), \(\xi\in \extp{3}{\R^n}\), the horizontal
composition in~\(\cG\) simply adds the \(\xi\)-components.  The unit
arrow in~\(\cG\) is \((0,0)\) and the left and right unitors are the
identity bigons \((0,0)\) because the morphism~\(F\) is strictly
unital.  The associator \(a(t_1,t_2,t_3)\) for \(t_1,t_2,t_3\in\R^n\)
is defined so that the following diagram of \(2\)-arrows commutes:
\[
\xymatrix@C+4em{
  F( (t_1,0)\cdot F((t_2,0)\cdot (t_3,0))) &
  F( F((t_1,0)\cdot (t_2,0))\cdot (t_3,0))
  \ar@{=>}[l]_{a(t_1,t_2,t_3)} \\
  (t_1,0)\cdot F((t_2,0)\cdot (t_3,0))
  \ar@{=>}[u]_{\Phi(t_1+t_2+t_3,t_1\wedge (t_2+t_3))} &
  F((t_1,0)\cdot (t_2,0))\cdot (t_3,0)
  \ar@{=>}[u]^{\Phi(t_1+t_2+t_3,(t_1+t_2)\wedge t_3)} \\
  (t_1,0)\cdot ((t_2,0)\cdot (t_3,0))
  \ar@{=>}[u]_{1_{(t_1,0)} \bullet \Phi(t_2+t_3,t_2\wedge t_3)}
  \ar@{=}[r] &
  ((t_1,0)\cdot (t_2,0))\cdot (t_3,0)
  \ar@{=>}[u]^{\Phi(t_1+t_2,t_1\wedge t_2) \bullet 1_{(t_3,0)}}
}
\]
Since this involves \(c_{(t_1,0)}(-t_2\wedge t_3,0) = (-t_2\wedge t_3,
-t_1\wedge t_2\wedge t_3)\), we get
\begin{multline*}
  a(t_1,t_2,t_3)
  = (-t_1\wedge (t_2+t_3),0) - (-(t_1+t_2)\wedge t_3,0)
  \\+ (-t_2\wedge t_3, -t_1\wedge t_2\wedge t_3)
  - (-t_1\wedge t_2,0)
  = (0,-t_1\wedge t_2\wedge t_3).
\end{multline*}
These computations lead us to the following definition:

\begin{definition}
  \label{def:G_cat}
  Let~\(\cG(*,*)\) be the Lie groupoid given by the constant bundle
  of Abelian Lie groups \(\extp{3}{\R^n}\) over~\(\R^n\).  Let~$\cG$
  be the weak topological $2$-group with one object~$*$
  and~\(\cG(*,*)\) as its groupoid of arrows and bigons.  The
  composition functor $\cdot \colon \cG(*,*) \times \cG(*,*) \to
  \cG(*,*)$ is addition in both components.  The unitors $l$ and~$r$
  (called identity transformations in~\cite{Benabou:Bicategories})
  are trivial.  The associator (called associativity isomorphism
  in~\cite{Benabou:Bicategories}) is
  \[
  \alpha \colon \R^n \times \R^n \times \R^n
  \to \R^n \times \extp{3}{\R^n},\quad
  (t_1,t_2,t_3) \mapsto (t_1+t_2+t_3, -t_1 \wedge t_2 \wedge t_3).
  \]
\end{definition}

It is routine to check that~\(\cG\) is a weak $2$-category (see
\cite{Buss-Meyer-Zhu:Higher_twisted}*{Definition 2.7}) or a
bicategory in the notation of \cite{Benabou:Bicategories}*{(1.1)}.
All arrows and bigons are invertible and~\(\cG\) has only one
object, so it is a weak Lie \(2\)-group.

The computations above suggest that~\(\cG\) should be equivalent to
the strict \(2\)\nobreakdash-group~\(\cat[\cH]\) associated to the crossed
module~\(\cH\).  More precisely, we are going to construct morphisms
\(\iota\colon \cG\to\cat[\cH]\) and \(\pi\colon \cat[\cH]\to\cG\)
such that \(\pi\circ \iota\) is the identity morphism on~\(\cG\) and
\(\iota\circ\pi=\Ad{\Phi}\) is equivalent to the identity morphism
on~\(\cat[\cH]\) by the transformation~\(\Phi\).  The
morphism~\(\iota\) maps the arrow~\(t\) to~\((t,0)\) and \(\xi\colon
t\Rightarrow t\) to \((0,\xi)\colon (t,0)\Rightarrow (t,0)\).  This
is strictly unital, so the bigon \(\iota(1_*)\Rightarrow 1_*\) is
\((0,0)\).  The compatibility with multiplication is given by
\begin{multline*}
  \omega_\iota(t_1,t_2) = (-t_1\wedge t_2,0)\colon
  (t_1+t_2,t_1\wedge t_2)
  = \iota(t_1)\cdot \iota(t_2)
  \\\Rightarrow \iota(t_1+t_2)
  = (t_1+t_2,0).
\end{multline*}
Routine computations show that~\(\iota\) is a morphism; in
particular, the cocycle condition
\cite{Buss-Meyer-Zhu:Higher_twisted}*{(4.2)} for~\(\omega_\iota\)
becomes
\begin{multline*}
  c_{(t_1,0)}(\omega_\iota(t_2,t_3)) -
  \omega_\iota(t_1+t_2,t_3) +
  \omega_\iota(t_1,t_2+t_3) -
  \omega_\iota(t_1,t_2)
  \\- (0,-t_1\wedge t_2\wedge t_3) = 0,
\end{multline*}
which holds because of our choice of the associator in~\(\cG\).

The projection \(\pi\colon \cat[\cH]\to\cG\) is defined so that
\(\iota\circ\pi=\Ad{\Phi}\).  Since~\(\iota\) is faithful on arrows
and bigons, this determines~\(\pi\) uniquely, and it implies
that~\(\pi\) is a morphism because~\(\Ad{\Phi}\) is one.  Our Ansatz
dictates \(\pi(t,\eta)= t\) and
\[
\pi((\theta,\xi)\colon (t,\eta)\Rightarrow
(t,\eta+\theta)) = (\xi\colon t\Rightarrow t),
\]
so that \(\iota\circ\pi\) and~\(\Ad{\Phi}\) involve the same maps on
arrows and bigons.  The morphism~\(\pi\) must be strictly unital
because~\(\iota\) and~\(\Ad{\Phi}\) are so.  The canonical bigon
\[
\omega_{\iota\pi}((t_1,\eta_1),(t_2,\eta_2))\colon
\iota\pi(t_1,\eta_1)\cdot \iota\pi(t_2,\eta_2) \Rightarrow
\iota\pi((t_1,\eta_1)\cdot (t_2,\eta_2))
\]
is the vertical composite of
\begin{align*}
  \omega_\iota(\pi(t_1,\eta_1),\pi(t_2,\eta_2))&\colon
  \iota\pi(t_1,\eta_1)\cdot \iota\pi(t_2,\eta_2)\Rightarrow
  \iota(\pi(t_1,\eta_1)\cdot \pi(t_2,\eta_2)),\\
  \iota(\omega_\pi((t_1,\eta_1),(t_2,\eta_2)))&\colon
  \iota(\pi(t_1,\eta_1)\cdot \pi(t_2,\eta_2)) \Rightarrow
  \iota(\pi((t_1,\eta_1)\cdot (t_2,\eta_2))).
\end{align*}
So we must put \(\omega_\pi((t_1,\eta_1),(t_2,\eta_2)) = t_1\wedge
\eta_2\colon t_1+t_2\Rightarrow t_1+t_2\) to get \(\omega_{\iota\pi}
= \omega_{\Ad{\Phi}}\).  This finishes the construction of~\(\pi\).

Since \(\iota\pi=\Ad{\Phi}\), \(\Phi\) is a transformation from the
identity on~\(\cat[\cH]\) to~\(\iota\pi\).  The composite
\(\pi\iota\colon \cG\to\cG\) is the identity on objects and arrows,
strictly unital, and involves the identical
transformation~\(\omega_{\pi\iota}\), so it is equal to the identity
functor.  Thus \(\iota\) and~\(\pi\) are equivalences of weak
\(2\)-groups inverse to each other.  Both are given by smooth maps,
so we have an equivalence of Lie \(2\)-groups.

\begin{remark}
  \label{rem:quotient_of_ext3}
  In the arguments above, we may replace \(\extp{3}{\R^n}\)
  everywhere by \(\extp{3}{\R^n}/\Gamma\) for any closed
  subgroup~\(\Gamma\).  We shall be interested, in particular, in
  the case where we use \(\extp{3}{\R^n}/\extp{3}{\Z^n}\) because
  the latter group is compact and because the actions of \(\cG\)
  or~\(\cH\) that we need factor through this quotient in the case
  of \(\torus^n\)-bundles.
\end{remark}

\section{Equivariant Brauer groups for bigroupoids}
\label{sec:equivariant_Brauer}

The equivariant Brauer group~\(\Br{G}{P}\) of a transformation
group~\(G\ltimes P\) is defined
in~\cite{Crocker-Kumjian-Raeburn-Williams:Brauer}.  It classifies
continuous trace \(\mathrm{C}^*\)-algebras with spectrum~\(P\) with a
\(G\)-action that lifts the given action on~\(P\).  This definition
is extended from transformation groups to locally compact groupoids
in~\cite{Kumjian-Muhly-Renault-Williams:Brauer}.  Here we extend it
further to locally compact bigroupoids.  ``Bigroupoids'' are called
``weak \(2\)-groupoids'' in~\cite{Buss-Meyer-Zhu:Higher_twisted}.
The bigroupoids we need combine the groupoid~\(\cG\) defined in
Section~\ref{sec:crossed_module_extension} with an action
\(\alpha\colon \R^n\to\mathrm{Diff}(P)\) of~\(\R^n\) by diffeomorphisms on
a manifold~\(P\).

Explicitly, we consider the following locally compact
bigroupoid~\(\cat\).  It has object space \(\cat^0=P\), arrow space
\(\cat^1 = \R^n\times P\), and space of bigons \(\cat^2 =
\extp{3}{\R^n}\times \R^n\times P\).  Here an arrow \((t,p)\in
\R^n\times P\) has source~\(p\) and range \(\alpha_t(p)\), and the
multiplication is the usual one: \((t_1,\alpha_{t_2}(p_{2}))\cdot
(t_2,p_2) = (t_1+t_2,p_2)\); a bigon \((\xi,t,p)\in
\extp{3}{\R^n}\times \R^n\times P\) has source and range~\((t,p)\);
the vertical composition adds the \(\xi\)-components: \((\xi_1,t,p)
\cdot (\xi_2,t,p) = (\xi_1+\xi_2,t,p)\); and the horizontal
composition is
\[
(\xi_1,t_1,\alpha_{t_2}(p_{2}))\bullet (\xi_2,t_2,p_2)
= (\xi_1+\xi_2,t_1+t_2,p_2).
\]
The unit arrow on an object~\(p\) is \((0,p)\), the unit bigon on an
arrow~\((t,p)\) is \((0,t,p)\).  Units are strict, that is, the left
and right unitors are trivial.  The associator is
\[
a((t_1,p_1),(t_2,p_2),(t_3,p_3))
= (-t_1\wedge t_2\wedge t_3,t_1+t_2+t_3,p_3)
\]
for a triple of composable arrows; that is, for \(p_i\in P\), \(t_i\in
\R^n\) for \(i=1,2,3\) with \(p_{i-1}=\alpha_{t_i}(p_{i})\) for
\(i=2,3\).  It is routine to check that this is a bigroupoid.  All the
spaces are smooth manifolds, hence locally compact, all the operations
are smooth maps, hence continuous, and range and source maps are
surjective submersions, hence open.  Thus \(\cat=\cG\ltimes P\) is a
Lie bigroupoid and a locally compact bigroupoid.

Actions of locally compact bigroupoids on \(\mathrm{C}^*\)-algebras are
defined in~\cite{Buss-Meyer-Zhu:Higher_twisted}.  We shall use
\cite{Buss-Meyer-Zhu:Higher_twisted}*{Definition~4.1} for the
correspondence bicategory as target bicategory (``bicategories'' are
called ``weak \(2\)-categories''
in~\cite{Buss-Meyer-Zhu:Higher_twisted}), but with some changes.
First, we require functors to be strictly unital, that is, the
bigons~\(u_x\) in
\cite{Buss-Meyer-Zhu:Higher_twisted}*{Definition~4.1} are identities;
this is no restriction of generality because any functor is
equivalent to one with this property, as remarked
in~\cite{Buss-Meyer-Zhu:Higher_twisted}.  Secondly, since~\(\cat\)
has invertible arrows, all correspondences appearing in an action
of~\(\cat\) are equivalences, that is, imprimitivity bimodules.
Third, we need continuous actions; continuity is explained at the
end of Section~4.1 in~\cite{Buss-Meyer-Zhu:Higher_twisted}.  Fourth,
we work with the opposite bicategory of imprimitivity bimodules,
that is, an \(A,B\)-imprimitivity bimodule is viewed as an arrow
from~\(B\) to~\(A\); otherwise the multiplication maps below would
go from \(E_g\otimes E_h\) to~\(E_{hg}\) instead of~\(E_{gh}\).

Now we define a \emph{continuous action of a bigroupoid~\(\cat\) by
  correspondences} or, equivalently, by imprimitivity bimodules,
with the modifications mentioned above.  Such an action requires the
following data:
\begin{itemize}
\item a \(C_0(\cat^0)\)-\(\mathrm{C}^*\)-algebra~\(A\); we denote its fibres
  by~\(A_x\) for \(x\in \cat^0\);
\item a \(C_0(\cat^1)\)-linear imprimitivity bimodule~\(E\) between
  the pull-backs \(r^*(A)\) and~\(s^*(A)\) of~\(A\) along the range
  and source maps; thus~\(E\) is a bundle over~\(\cat^1\) where the
  fibre at \(g\in\cat^1\) is an \(A_{r(g)},A_{s(g)}\)-imprimitivity
  bimodule~\(E_g\);
\item isomorphisms \(\omega_{g,h}\colon E_g\otimes_{A_{s(g)}} E_h
  \to E_{gh}\) of imprimitivity bimodules for all \(g,h\in\cat^1\)
  with \(s(g)=r(h)\), which are continuous in the sense that
  pointwise application of~\(\omega_{g,h}\) gives an isomorphism of
  imprimitivity bimodules
  \[
  \omega\colon \pr_1^*(E) \otimes_{(s\,\pr_1)^*(A)} \pr_2^*(E)
  \to \mu^*(E),
  \]
  where \(\pr_1,\pr_2,\mu\) are the continuous maps that map a
  pair~\((g,h)\) of composable arrows to \(g\), \(h\) and~\(gh\),
  respectively; so \(s\,\pr_1=r\,\pr_2\) maps \((g,h)\) to
  \(s(g)=r(h)\);
\item isomorphisms of imprimitivity bimodules \(U_b\colon E_{s_2(b)}
  \to E_{r_2(b)}\) for all bigons \(b\in\cat^2\), which are
  continuous in the sense that they give an isomorphism of
  imprimitivity bimodules \(U\colon s_2^*(E) \to r_2^*(E)\); here
  \(r_2,s_2\colon \cat^2\rightrightarrows\cat^1\) map a bigon to its range and
  source arrow;
\end{itemize}
this data is subject to the following conditions:
\begin{enumerate}[label=(A\arabic*)]
\item \label{enum:units_E} \(E_{1_x}=A_x\) for all \(x\in \cat^0\),
  and the restriction of~\(E\) to units is~\(A\);
\item \label{enum:units_omega} \(\omega_{g,1}\colon
  E_g\otimes_{A_{s(g)}} A_{s(g)} \to E_g\) and \(\omega_{1,h}\colon
  A_{r(h)} \otimes_{A_{r(h)}} E_h \to E_h\) are the canonical
  isomorphisms for all \(g,h\in\cat^1\);
\item \label{enum:U_vertical} \(U\) is multiplicative for vertical
  products: \(U_{b_1} \circ U_{b_2} = U_{b_1 \cdot b_2}\) for
  vertically composable bigons \(b_1,b_2\in\cat^2\);
\item \label{enum:U_horizontal} if \(b_1\colon f_1\Rightarrow g_1\)
  and \(b_2\colon f_2\Rightarrow g_2\) are horizontally composable
  bigons, that is, \(s(f_1) = s(g_1) = r(f_2) = r(g_2)\), then the
  following diagram commutes:
  \[
  \xymatrix{
    E_{f_1} \otimes_{A_{s(f_1)}} E_{f_2}
    \ar[d]_-{U_{b_1}\otimes U_{b_2}} \ar[r]^-{\omega_{f_1,f_2}} &
    E_{f_1f_2} \ar[d]^-{U_{b_1 \bullet b_2}} \\
    E_{g_1} \otimes_{A_{s(g_1)}} E_{g_2}
    \ar[r]^-{\omega_{g_1,g_2}} &
    E_{g_1 g_2}
  }
  \]
\item \label{enum:omega_cocycle} if \(g_1,g_2,g_3\) are composable
  arrows in~\(\cat^1\), then the following diagram commutes:
  \[
  \xymatrix{
    (E_{g_1}\otimes E_{g_2}) \otimes E_{g_3} \ar@{<->}[r]
    \ar[d]^{\omega_{g_1,g_2}\otimes1}&
    E_{g_1}\otimes (E_{g_2} \otimes E_{g_3})
    \ar[d]^{1\otimes\omega_{g_2,g_3}}\\
    E_{g_1 g_2} \otimes E_{g_3}
    \ar[d]^{\omega_{g_1g_2,g_3}}&
    E_{g_1}\otimes E_{g_2 g_3}
    \ar[d]^{\omega_{g_1, g_2g_3}}\\
    E_{(g_1 g_2) g_3} \ar[r]^{U_{a(g_1,g_2,g_3)}}&
    E_{g_1 (g_2 g_3)}
  }
  \]
  here \(a(g_1,g_2,g_3)\colon (g_1 g_2) g_3 \to g_1 (g_2 g_3)\) is
  the associator of~\(\cat^2\), and we dropped the
  subscripts~\(A_?\) on~\(\otimes\) to avoid clutter.
\end{enumerate}
Condition~\ref{enum:U_vertical} says that the maps \(g\mapsto E_g\)
and \(b\mapsto U_b\) form a functor; \ref{enum:U_horizontal} says
that the maps~\(\omega_{g_1,g_2}\) are natural with respect to
bigons; \ref{enum:units_E} says that our functor is strictly unital;
\ref{enum:units_omega} is equivalent to the coherence conditions
\cite{Buss-Meyer-Zhu:Higher_twisted}*{(4.3)} for unitors;
\ref{enum:omega_cocycle} is
\cite{Buss-Meyer-Zhu:Higher_twisted}*{(4.2)}.

To define the equivariant Brauer group of~\(\cat\), we also need
equivalences between such actions.  Let \((A^1,E^1,\omega^1,U^1)\)
and \((A^2,E^2,\omega^2,U^2)\) be continuous actions of~\(\cat\).  A
\emph{transformation} between them consists of the following data:
\begin{itemize}
\item a \(C_0(\cat^0)\)-linear correspondence~\(F\) between
  \(A^1\) and~\(A^2\), with fibres~\(F_x\) for \(x\in\cat^0\);
\item isomorphisms of correspondences
  \[
  V_g\colon E^1_g \otimes_{A^1_{s(g)}} F_{s(g)}
  \to F_{r(g)} \otimes_{A^2_{r(g)}} E^2_g
  \]
  for all \(g\in \cat^1\), which are continuous in the sense that
  they give an isomorphism \(V\colon E^1 \otimes_{s^*(A^1)} s^*(F)
  \to r^*(F) \otimes_{r^*(A^2)} E^2\);
\end{itemize}
this must satisfy the following conditions:
\begin{enumerate}[label=(T\arabic*)]
\item \label{enum:trafo_natural} for each bigon \(b\colon
  g\Rightarrow h\), the following diagram commutes:
  \[
  \xymatrix{
    E^1_g \otimes_{A^1_{s(g)}} F_{s(g)}
    \ar[d]_-{U^1_b\otimes 1} \ar[r]^-{V_g} &
    F_{r(g)} \otimes_{A^2_{r(g)}} E^2_g
    \ar[d]^-{1\otimes U^2_b} \\
    E^1_h \otimes_{A^1_{s(h)}} F_{s(h)} \ar[r]^-{V_h} &
    F_{r(h)} \otimes_{A^2_{r(h)}} E^2_h
  }
  \]
\item \label{enum:trafo_unit} \(V_{1_x}\colon A^1_x \otimes_{A^1_x}
  F_x \to F_x \otimes_{A^2_x} A^2_x\) is the canonical isomorphism;
\item \label{enum:trafo_cocycle} the following diagram commutes for
  \(g,h\in \cat^1\) with \(s(g)=r(h)\):
  \[
  \xymatrix{
    E^1_g \otimes_{A^1_{s(g)}} E^1_h \otimes_{A^1_{s(h)}} F_{s(h)}
    \ar[r]^-{\omega^1_{g,h}\otimes 1}
    \ar[dd]_-{1\otimes V_h} &
    E^1_{gh} \otimes_{A^1_{s(h)}} F_{s(h)}
    \ar[d]^-{V_{gh}} \\
    &F_{r(g)} \otimes_{A^2_{r(g)}} E^2_{gh} \\
    E^1_g \otimes_{A^1_{s(g)}} F_{s(g)} \otimes_{A^2_{s(g)}} E^2_h
    \ar[r]^-{V_g\otimes 1} &
    F_{r(g)} \otimes_{A^2_{r(g)}} E^2_g \otimes_{A^2_{s(g)}} E^2_h
    \ar[u]_-{1\otimes \omega^2_{g,h}}
  }
  \]
\end{enumerate}
Condition~\ref{enum:trafo_natural} says that \(V\) is natural with
respect to bigons, \ref{enum:trafo_unit} is the coherence condition
for units and \ref{enum:trafo_cocycle} is the coherence condition
for multiplication.

An \emph{equivalence} between two actions is a transformation where
each~\(F_x\) or, equivalently, \(F\) is an imprimitivity bimodule;
then the maps~\(V_g\) and~\(V\) are automatically isomorphisms of
imprimitivity bimodules, that is, compatibility with the left inner
product comes for free.

A \emph{modification} between two transformations
\[
(F^1,V^1), (F^2,V^2)\colon
(A^1,E^1,\omega^1,U^1) \rightrightarrows (A^2,E^2,\omega^2,U^2)
\]
is given by the following data:
\begin{itemize}
\item isomorphisms of correspondences \(W_x\colon F^1_x\to F^2_x\)
  for all \(x\in\cat^0\) that give an isomorphism \(W\colon F^1\to
  F^2\);
\end{itemize}
this must satisfy the following condition:
\begin{enumerate}[label=(M\arabic*)]
\item for each \(g\in\cat^1\) the following diagram commutes:
  \[
  \xymatrix{
    E^1_g \otimes_{A^1_{s(g)}} F^1_{s(g)}
    \ar[d]_-{1\otimes W_{s(g)}} \ar[r]^-{V^1_g} &
    F^1_{r(g)} \otimes_{A^2_{r(g)}} E^2_g
    \ar[d]^-{W_{r(g)}\otimes 1} \\
    E^1_g \otimes_{A^1_{s(g)}} F^2_{s(g)} \ar[r]^-{V^2_g} &
    F^2_{r(g)} \otimes_{A^2_{r(g)}} E^2_g
  }
  \]
\end{enumerate}
This is the obvious notion of isomorphism between two
transformations.

Two actions of~\(\cat\) may be tensored together in the obvious way,
using the fibrewise maximal tensor product as
in~\cite{Huef-Raeburn-Williams:Brauer_semigroup}.

\begin{definition}
  \label{def:Brauer}
  The \emph{equivariant Brauer group} \(\Br{}{\cat}\) of the locally
  compact bigroupoid~\(\cat\) is the set of equivalence classes of
  continuous actions of~\(\cat\) on continuous trace
  \(\mathrm{C}^*\)-algebras with spectrum~\(\cat^0\); the group structure is
  the tensor product over~\(\cat^0\).  We also write
  \(\Br{\cG}{P}=\Br{}{\cG\ltimes P}\).
\end{definition}

The usual proof in~\cite{Crocker-Kumjian-Raeburn-Williams:Brauer}
that this defines an Abelian group carries over to our case.  The
formula for the multiplication is particularly easy:
\[
(A_1,E_1,\omega_1,U_1) \otimes_P (A_2,E_2,\omega_2,U_2)
= (A_1\otimes_P A_2,E_1\otimes_P E_2,\omega_1\otimes_P \omega_2,U_1
\otimes_P U_2),
\]
where~\(\otimes_P\) means the maximal fibrewise tensor product of
\(\mathrm{C}^*\)-algebras over~\(P\), the corresponding external tensor product
of imprimitivity bimodules, or the external tensor product of
operators.  Similar formulas work for equivalences between actions, so
that the operation~\(\otimes_P\) descends to equivalence classes.  The
usual symmetry of~\(\otimes_P\) gives that this multiplication is
commutative.  Further details are left to the reader.

\begin{definition}
  \label{def:Picard}
  The \emph{equivariant Picard group} \(\Pic{\cat}{A,E,\omega,U}\) of
  an action of~\(\cat\) is the group of all equivalence classes of
  self-equivalences on \((A,E,\omega,U)\), where two self-equivalences
  are considered equivalent if there is a modification between them.
  The group structure is the (vertical) composition of transformations.
\end{definition}

In the special case of actions on continuous trace \(\mathrm{C}^*\)-algebras,
the definitions above simplify because of the following lemma:

\begin{lemma}
  \label{lem:equivalence_continuous-trace}
  Let \(A\) and~\(B\) be continuous trace \(\mathrm{C}^*\)-algebras with
  spectrum~\(X\).  Let \(E_1\) and~\(E_2\) be two \(C_0(X)\)-linear
  equivalences from~\(A\) to~\(B\).  There is a Hermitian complex line bundle
  \(L\) over~\(X\) with \(E_2 \cong E_1\otimes_X L\), and conversely
  \(E_1\otimes_X L\) is another equivalence from \(A\) to~\(B\) for
  any Hermitian complex line bundle~\(L\).

  Let \(U_1,U_2\colon E_1\rightrightarrows E_2\) be two isomorphisms
  of equivalences.  There is a continuous map \(\varphi\colon X\to
  \torus\) with \(U_2(x) = \varphi(x)\cdot U_1(x)\) for all \(x\in
  X\), and conversely \(U_1\cdot\varphi\) for a continuous map
  \(\varphi\colon X\to \torus\) is another isomorphism \(E_1 \to
  E_2\).
\end{lemma}

We always identify a Hermitian complex line bundle with its space of
\(C_0\)-sections, which is a \(C_0(P),C_0(P)\)-imprimitivity bimodule.

\begin{proof}
  Let~\(E_2^*\) be the inverse equivalence.  Then \(E = E_1\otimes_B
  E_2^*\) is a \(C_0(X)\)-linear self-equivalence of~\(A\).  We are
  going to prove that~\(E\)
  is of the form \(A\otimes_X L\) for a Hermitian complex line bundle~\(L\)
  over~\(X\).  The opposite algebra~\(A^\op\) is an inverse
  for~\(A\) in the Brauer group, that is, \(A\otimes_X A^\op \cong
  C_0(X,\K)\); this is Morita equivalent to~\(C_0(X)\).  It is
  well-known that a \(C_0(X)\)-linear self-equivalence of~\(C_0(X)\)
  is the same as a Hermitian complex line bundle over~\(X\).  Since \(C_0(X)\)
  and~\(C_0(X,\K)\) are \(C_0(X)\)-linearly equivalent, they have
  isomorphic groups of \(C_0(X)\)-linear self-equivalences.  Thus
  \(E\otimes_X A^\op\) is the space \(L\otimes \K\), where~\(L\) is
  the space of sections of a Hermitian complex line bundle over~\(X\).  On the
  one hand, \(E\otimes_X (A^\op\otimes_X A) \cong E\otimes_X
  C_0(X,\K)\) is just the stabilisation of~\(E\); on the other hand,
  it is \((E\otimes_X A^\op)\otimes_X A \cong L\otimes\K \otimes_X
  A\).  Due to the \(C_0(X)\)-linear equivalence between \(A\)
  and~\(A\otimes\K\), we may remove the stabilisations again to see
  that \(E\cong A\otimes_X L\).  It is clear, conversely, that
  \(E\otimes_X L\) is again a \(C_0(X)\)-linear self-equivalence
  of~\(A\).

  If \(f_1,f_2\colon E_1\to E_2\) are two isomorphisms of
  imprimitivity bimodules, then \(f_2^{-1}f_1\colon E_1\to E_1\) is
  an isomorphism.  In each fibre, \((E_1)_x\) is an imprimitivity
  bimodule from \(\K\) to~\(\K\).  The isomorphism \(f_2^{-1}f_1\)
  gives a unitary bimodule map, and any such map is just
  multiplication with a scalar of absolute value~\(1\).  This gives
  the function \(\varphi\colon X\to\torus\), which is continuous
  because \(U_1\) and~\(U_2\) are continuous.
\end{proof}

\begin{proposition}
  \label{pro:Picard_constant}
  The groups \(\Pic{\cat}{A,E,\omega,U}\) are canonically isomorphic
  for all actions \((A,E,\omega,U)\) of~\(\cat\) on continuous-trace
  \(\mathrm{C}^*\)-algebras over~\(\cat^0\).
\end{proposition}

\begin{proof}
  Let \((A_i,E_i,\omega_i,U_i)\) for \(i=1,2\) be two actions
  of~\(\cat\) on continuous-trace \(\mathrm{C}^*\)-algebras over~\(\cat^0\).
  Since the Brauer group has inverses, there are other actions
  \((A_i,E_i,\omega_i,U_i)\), \(i=3,4\), of~\(\cat\) on
  continuous-trace \(\mathrm{C}^*\)-algebras over~\(\cat^0\) such that
  \begin{align*}
    (A_1,E_1,\omega_1,U_1) \otimes_{\cat^0} (A_3,E_3,\omega_3,U_3)
    &\cong (A_2,E_2,\omega_2,U_2),\\
    (A_2,E_2,\omega_2,U_2) \otimes_{\cat^0} (A_4,E_4,\omega_4,U_4)
    &\cong (A_1,E_1,\omega_1,U_1).
  \end{align*}
  Tensoring a self-equivalence of \((A_1,E_1,\omega_1,U_1)\) with the
  identity equivalence of \((A_3,E_3,\omega_3,U_3)\) gives a
  self-equivalence of \((A_2,E_2,\omega_2,U_2)\).  Tensoring a
  self-\hspace{0pt}equivalence of \((A_2,E_2,\omega_2,U_2)\) with the
  identity equivalence of \((A_4,E_4,\omega_4,U_4)\) gives a
  self-equivalence of \((A_1,E_1,\omega_1,U_1)\).  This defines group
  homomorphisms between the two Picard groups that are inverse to each
  other because \((A_3,E_3,\omega_3,U_3)\) and
  \((A_4,E_4,\omega_4,U_4)\) are inverse to each other in the Brauer
  group.
\end{proof}

We denote the common Picard group of all actions of~\(\cat\) on
continuous-trace \(\mathrm{C}^*\)-algebras over~\(\cat^0\) by \(\Pic{}{\cat}\)
and also write \(\Pic{\cG}{P} = \Pic{}{\cG\ltimes P}\).

Lemma~\ref{lem:equivalence_continuous-trace} becomes even more
powerful when we use that the functors that send a paracompact
space~\(X\) to the set of equivalence classes of Hermitian complex line bundles
or continuous trace \(\mathrm{C}^*\)-algebras with spectrum~\(X\) are both
homotopy invariant; actually, they are \(H^2(X,\Z)\) and
\(H^3(X,\Z)\).

\begin{lemma}
  \label{lem:E_omega_U_exist}
  Assume that the spaces \(\cat^i\) are paracompact.  Assume that
  there is a continuous homotopy \(H\colon \cat^1\times[0,1]\to
  \cat^1\) with \(H_0=\id{\cat^1}\), \(H_1 = u\circ r\) for the unit
  and range maps \(u,r\) between \(\cat^1\) and~\(\cat^0\), and \(r
  \circ H_t =r\) for all \(t\in[0,1]\).  Assume further that
  \(r_2\colon \cat^2\to\cat^1\) is a homotopy equivalence.
  Let~\(A\) be any continuous trace \(\mathrm{C}^*\)-algebra with
  spectrum~\(\cat^0\).  Then:
  \begin{enumerate}
  \item \label{enum:E_exists} there is an imprimitivity
    bimodule~\(E\) between \(r^*(A)\) and~\(s^*(A)\) that restricts
    to the identity on units;
  \item \label{enum:E_unique} any two such~\(E\) are isomorphic with
    an isomorphism that is the identity on units;
  \item \label{enum:omega_exists} there are isomorphisms
    \[
    \omega\colon \pr_1^*(E) \otimes_{(s\, \pr_1)^*(A)} \pr_2^*(E)
    \to \mu^*(E)\quad\text{and}\quad
    U\colon s_2^*(E) \to r_2^*(E)
    \]
    such that \(\omega_{1_{r(g)},g}\) and~\(\omega_{g,1_{s(g)}}\) are
    the canonical isomorphisms and~\(U_{1_g}\) is the identity for all
    \(g\in\cat^1\);
  \item \label{enum:omega_unique} any two choices for \(\omega\)
    and~\(U\) as above differ by pointwise multiplication with
    \(\exp(2\pi\ima \varphi)\) for a continuous map \(\varphi\colon
    \cat^1 \times_{s,\cat^0,r} \cat^1 \sqcup \cat^2\to \R\) with
    \(\varphi(1_{r(g)},g)=0\), \(\varphi(g,1_{s(g)})=0\) and
    \(\varphi(1_g)=0\) for all \(g\in\cat^1\).
  \end{enumerate}
  Let \((A^i,E^i,\omega^i,U^i)\) for \(i=1,2\) be actions
  of~\(\cat\) and let~\(F\) be an \(A^1,A^2\)-imprimitivity
  bimodule.  Then:
  \begin{enumerate}[resume]
  \item \label{enum:V_exists} there is an isomorphism of
    imprimitivity bimodules \(V\colon E^1 \otimes_{s^*(A^1)} s^*(F)
    \cong r^*(F) \otimes_{r^*(A^2)} E^2\) that restricts to the
    canonical isomorphism on units;
  \item \label{enum:V_unique} any two choices \(V^1\) and~\(V^2\) as
    in~\ref{enum:V_exists} differ by pointwise multiplication with
    \(\exp(2\pi\ima \varphi)\) for a continuous map \(\varphi\colon
    \cat^1\to \R\) with \(\varphi(1_x)=0\) for all \(x\in\cat^0\).
  \end{enumerate}
\end{lemma}

\begin{proof}
  The pull-backs \(s^*(A)\) and~\(r^*(A)\) are continuous trace
  \(\mathrm{C}^*\)-algebras over~\(\cat^1\) that restrict to the same
  continuous trace \(\mathrm{C}^*\)-algebra on \(u(\cat^0)\subseteq \cat^1\).
  By assumption, \(\cat^0\) is a deformation retract of~\(\cat^1\),
  and the functor \(X\mapsto \Br{}{X}\) is homotopy invariant on
  paracompact spaces.  Since \(u^* s^*(A) \cong u^* r^*(A)\), there
  must be an imprimitivity bimodule~\(E\) between \(s^*(A)\)
  and~\(r^*(A)\).  The restriction of~\(E\) to units differs from
  the identity imprimitivity bimodule by some line bundle by
  Lemma~\ref{lem:equivalence_continuous-trace}.  This line bundle
  extends to a line bundle over~\(\cat^1\) because \(\cat^0\) is a
  deformation retract of~\(\cat^1\).  Tensoring~\(E\) with the
  opposite of that line bundle, we can arrange that \(u^*(E)\) is
  isomorphic to the identity equivalence on~\(A\); then we may
  replace~\(E\) by an isomorphic equivalence such that \(u^*(E)\) is
  \emph{equal} and not just isomorphic to the identity equivalence
  on~\(A\).  This proves~\ref{enum:E_exists}.

  Let \(E^1\) and~\(E^2\) be two equivalences between \(s^*(A)\)
  and~\(r^*(A)\) that restrict to the identity on units.  Then
  \(E^2\cong E^1\otimes_{\cat^1} L\) for some line bundle~\(L\) by
  Lemma~\ref{lem:equivalence_continuous-trace}.  Since \(u^*(E^2)=
  u^*(E^1)\), \(u^*(L)\) is trivialisable.  Since~\(\cat^0\) is a
  deformation retract of~\(\cat^1\), this implies that \(L\) is
  trivialisable over~\(\cat^1\), so \(E^1\cong E^2\).  This isomorphism
  restricted to units differs from the identity isomorphism by
  pointwise multiplication with some function \(\psi\colon
  \cat^0\to\torus\) by Lemma~\ref{lem:equivalence_continuous-trace}.
  Since the unit map is a homotopy equivalence, this function extends
  to \(\bar\psi\colon \cat^1\to\torus\).  Multiplying pointwise
  with~\(\bar\psi^{-1}\) gives another isomorphism \(E^1\cong E^2\)
  that restricts to the identity on units.  This
  proves~\ref{enum:E_unique}.

  We claim that the inclusion of the subspace
  \[
  X = \{(g,1_{s(g)})\mid g\in\cat^1\}
  \cup \{(1_{r(h)},h)\mid h\in\cat^1\}
  \]
  into \(\cat^1\times_{s,\cat^0,r} \cat^1\) is a homotopy equivalence.
  Indeed, the maps \((g,h)\mapsto (g,H_t(h))\) deformation-retract
  \(\cat^1\times_{s,\cat^0,r} \cat^1\) to the subspace of pairs
  \((g,1_{s(g)})\), and they restrict to a deformation retraction from
  the subspace~\(X\) onto the same space.

  The \(C_0(\cat^1\times_{s,r} \cat^1)\)-linear imprimitivity
  bimodules \(\pr_1^*(E) \otimes_{(s\,\pr_1)^*(A)} \pr_2^*(E)\)
  and~\(\mu^*(E)\) differ by some line bundle by
  Lemma~\ref{lem:equivalence_continuous-trace}.  Since the two
  imprimitivity bimodules are canonically isomorphic on the
  subspace~\(X\), the line bundle is trivial on~\(X\).  Since the
  inclusion of~\(X\) is a homotopy equivalence, the line bundle is
  trivial everywhere.  Hence there is an isomorphism~\(\omega\)
  between \(\pr_1^*(E) \otimes_{(s\,\pr_1)^*(A)} \pr_2^*(E)\)
  and~\(\mu^*(E)\).  On~\(X\), we also have the canonical isomorphism,
  which differs from~\(\omega\) by some continuous function
  \(\psi\colon X\to\torus\) by
  Lemma~\ref{lem:equivalence_continuous-trace}.  As above, we may
  correct the isomorphism~\(\omega\) so that it restricts to the
  identity on~\(X\) because~\(\psi\) extends continuously to
  \(\cat^1\times_{s,r} \cat^1\).

  A similar argument gives an isomorphism \(U\colon s_2^*(E)\to
  r_2^*(E)\) over the space of bigons~\(\cat^2\) with
  \(U(1_g)=\id{E_g}\) for all \(g\in\cat^1\) because~\(r_2\) is a
  homotopy equivalence.  This proves~\ref{enum:omega_exists}.

  Any two choices for \(\omega\) and~\(U\) differ through pointwise
  multiplication with some functions \(\cat^1\times_{s,r}
  \cat^1\to\torus\) and \(\cat^2\to\torus\) by
  Lemma~\ref{lem:equivalence_continuous-trace}; these functions
  are~\(1\) on~\(X\) or on unit bigons by our normalisations.  Since
  the inclusions of~\(X\) and unit bigons are homotopy equivalences,
  covering space theory allows to lift such a function to~\(\R\) as
  required in~\ref{enum:omega_unique}.

  The proofs of \ref{enum:V_exists} and~\ref{enum:V_unique} use the
  same ideas.  An isomorphism~\(V\) exists because the two
  imprimitivity bimodules are isomorphic on units and the inclusion
  of units is a homotopy equivalence, and it may be arranged to be
  the canonical isomorphism on units because any continuous function
  \(u(\cat^0)\to\torus\) extends to a continuous function
  \(\cat^1\to\torus\).  Lemma~\ref{lem:equivalence_continuous-trace}
  shows that two isomorphisms \(V_1\) and~\(V_2\) differ by a
  function \(\cat^1\to\torus\), which is constant equal to~\(1\) on
  units.  Any such function lifts to~\(\R\),
  giving~\ref{enum:V_unique}.
\end{proof}

Let \((A,E,\omega,U)\) be a continuous action of~\(\cat\).  If~\(F\)
is a \(C_0(\cat^0)\)-linear Morita equivalence from \(A\) to some
other \(\mathrm{C}^*\)-algebra~\(A'\), then we may transfer the action
from~\(A\) to~\(A'\) along~\(F\): let \(E'= F\otimes_A E \otimes_A
F^*\) and translate \(\omega\) and~\(U\) accordingly.  Hence up to
equivalence of actions, only the equivalence class of~\(A\) matters.
Similarly, if \(A\) is fixed and~\(E'\) is another equivalence
\(s^*(A)\cong r^*(A)\) with \(E\cong E'\), then we may use the
isomorphism \(E\cong E'\) to transfer \(\omega\) and~\(U\) to~\(E'\).
So in the definition of the Brauer group, only the Morita
equivalence class of~\(A\) and, for fixed~\(A\), the isomorphism
class of~\(E\) matter.

\section{Lifting actions to continuous trace algebras}
\label{sec:lifting_continuous-trace}

Let~\(P\) be a second countable, locally compact space with a
continuous action of~\(\R^n\); then~\(P\) is paracompact.
Lemma~\ref{lem:E_omega_U_exist} applies to the transformation groupoid
\(\R^n\ltimes P\) and the transformation bigroupoid \(\cG\ltimes P\)
because the Lie groups \(\R^n\) and \(\extp{3}{\R^n}\) are
contractible.  We use this to analyse the obstruction to lifting an
\(\R^n\)-action from~\(P\) to a continuous trace \(\mathrm{C}^*\)-algebra
over~\(P\).  Our results for \(\R^n\ltimes P\) are already contained
in~\cite{Crocker-Kumjian-Raeburn-Williams:Brauer}.

Consider the case~\(\R^n\) first.  Here there are no bigons, so the
datum~\(U\) is not there and the conditions
\ref{enum:U_vertical}--\ref{enum:U_horizontal} in the definition of
an action are empty.  Let~\(A\) be a continuous trace
\(\mathrm{C}^*\)-algebra over~\(P\).  Lemma~\ref{lem:E_omega_U_exist}
provides the data \(E\) and~\(\omega\) for an action, satisfying
\ref{enum:units_E} and~\ref{enum:units_omega}, but not yet
satisfying the cocycle condition~\ref{enum:omega_cocycle}.
Since~\(E\) is unique up to isomorphism, its choice does not affect
whether or not there is~\(\omega\)
satisfying~\ref{enum:omega_cocycle}, nor the resulting element in
the equivariant Brauer group.  By Lemma~\ref{lem:E_omega_U_exist},
any two choices for~\(\omega\) differ through pointwise
multiplication with a function of the form \(\exp(2\pi\ima
\varphi)\) for a continuous function \(\varphi \colon \R^n\times\R^n\times
P\to\R\); here we have identified the space of composable arrows
in~\(\R^n\ltimes P\) with \(\R^n\times\R^n\times P\).

Similarly, the space of composable \(k\)-tuples of arrows
in~\(\R^n\ltimes P\) is \((\R^n)^k\times P\).  By
Lemma~\ref{lem:equivalence_continuous-trace}, the two isomorphisms
\((E_{g_1}\otimes E_{g_2})\otimes E_{g_3}\to E_{g_1 g_2 g_3}\)
in~\ref{enum:omega_cocycle} differ by pointwise multiplication with a
function \((\R^n)^3\times P\to\torus\).  We can also view this as the
difference between the isomorphism \((E_{g_1} \otimes E_{g_2}) \otimes
E_{g_3} \to E_{g_1} \otimes (E_{g_2} \otimes E_{g_3})\) induced by the
two isomorphisms above and the canonical one, which makes it clear
that this function plays the role of an associator.  It is~\(1\) if
one of the \(\R^n\)-entries is~\(0\) by our normalisations.  Hence
arguments as in the proof of Lemma~\ref{lem:E_omega_U_exist} show that
it lifts uniquely to an \(\R\)-valued function
\(\psi\colon(\R^n)^3\times P\to\R\) that is~\(0\) if one of the
\(\R^n\)-entries is~\(0\).  When we multiply~\(\omega\) pointwise by
\(\exp(2\pi\ima\varphi)\), then this adds the following function
to~\(\psi\):
\[
\partial\varphi(t_1,t_2,t_3,p) =
\varphi(t_2,t_3,p) - \varphi(t_1+t_2,t_3,p)
+ \varphi(t_1,t_2+t_3,p) - \varphi(t_1,t_2,\alpha_{t_3}(p)).
\]
Since the associator diagram
\[
\xymatrix@R=0.5cm{
  ((E_{g_1} \otimes E_{g_2}) \otimes E_{g_3}) \otimes E_{g_4}
  \ar[r] \ar[dd] &
  (E_{g_1} \otimes (E_{g_2} \otimes E_{g_3})) \otimes E_{g_4}
  \ar[d] \\ &
  E_{g_1} \otimes ((E_{g_2} \otimes E_{g_3}) \otimes E_{g_4})
  \ar[d] \\
  (E_{g_1} \otimes E_{g_2}) \otimes (E_{g_3} \otimes E_{g_4})
  \ar[r] &
  E_{g_1} \otimes (E_{g_2} \otimes (E_{g_3} \otimes E_{g_4}))
}
\]
commutes, we deduce that~\(\psi\) automatically satisfies the cocycle
condition
\begin{multline*}
  0 = \partial\psi(t_1,t_2,t_3,t_4,p) =
  \psi(t_2,t_3,t_4,p)
  -\psi(t_1+t_2,t_3,t_4,p)
  \\+\psi(t_1,t_2+t_3,t_4,p)
  -\psi(t_1,t_2,t_3+t_4,p)
  +\psi(t_1,t_2,t_3,\alpha_{t_4}(p)).
\end{multline*}
More precisely, the function \(\exp(2\pi\ima\partial\psi)\) is
automatically the constant function~\(1\), and this implies the
above because \(\partial\psi\) vanishes if one of the \(t_i\)
is~\(0\) and the lifting of \(\torus\)-valued functions to
normalised \(\R\)-valued functions is unique if it exists.

As a result, when we view~\(\psi\) as a function from~\((\R^n)^3\)
to the Fr\'{e}chet space~\(C(P,\R^n)\) of continuous functions
\(P\to\R^n\) with the action of~\(\R^n\) induced from the action
on~\(P\), then \(\psi\) is a continuous \(3\)-cocycle; and we may
choose~\(\omega\) to satisfy the cocycle
condition~\ref{enum:omega_cocycle} if and only if this \(3\)-cocycle
is a coboundary.  Thus the action of~\(\R^n\) on~\(P\) lifts to an
action on~\(A\) if and only if the class of~\(\psi\) in the
continuous group cohomology \(H^3_\mathrm{cont}(\R^n,C(P,\R))\)
vanishes.  We call this class the \emph{lifting obstruction} of the
continuous trace \(\mathrm{C}^*\)-algebra~\(A\).

The Packer--Raeburn Stabilisation Trick shows that any action
of~\(\R^n\) by equivalences as above is equivalent to a strict
action by automorphisms on the stabilisation of~\(A\) (this is
contained in \cite{Buss-Meyer-Zhu:Higher_twisted}*{Theorem 5.3}).
Since stabilisation does not change the class of~\(A\) in the Brauer
group, we may assume that~\(A\) is stable.  Then the lifting
obstruction is the obstruction to the existence of a strict action
of~\(\R^n\) by automorphisms.  An obstruction for this is also
constructed in~\cite{Crocker-Kumjian-Raeburn-Williams:Brauer}, but
in the measurable group cohomology \(H^3_\mathrm{M}(\R^n,C(P,\R))\).
These two cohomology groups coincide by
\cite{Wigner:Algebraic_cohomology}*{Theorem~3}.

The continuous cohomology group \(H^3_\mathrm{cont}(\R^n,C(P,\R))\)
may be simplified if the action of~\(\R^n\) factors through a
torus~\(\torus^n\).  Let \(\extpd{k}{\R^n} = (\extp{k}{\R^n})^*\)
denote the vector space of antisymmetric \(k\)-linear maps
\((\R^n)^k\to \R\).

\begin{proposition}
  \label{pro:simplify_H3}
  Let~\(P\) be a second countable, locally compact \(\torus^n\)-space,
  viewed as an \(\R^n\)-space, with orbit space~\(P/\R^n\).  Then
  \[
  H^k_\mathrm{cont}(\R^n,C(P,\R))
  \cong C(P/\R^n, \extpd{k}{\R^n})
  \]
  for all \(k\ge0\).  The isomorphism maps \(\chi\colon P/\R^n
  \to \extpd{k}{\R^n}\) to the cocycle given by
  \[
  (\R^n)^k\times P\to\R,\qquad
  (t_1,\dotsc,t_k,p)\mapsto
  \chi([p])(t_1\wedge \dotsc\wedge t_k).
  \]
\end{proposition}

\begin{proof}
  If the action of~\(\torus^n\) on~\(P\) is free, then this is
  \cite{Mathai-Rosenberg:T-duality}*{Lemma 2.1}.  We explain why the
  result remains true for non-free actions of~\(\torus^n\).
  Throughout this proof, group cohomology is understood to be
  continuous group cohomology.

  Continuous representations of~\(\R^n\) on Fr\'echet spaces
  over~\(\R\) are equivalent to non-degenerate modules over the Banach
  algebra~\(L^1(\R^n)\) of integrable functions \(\R^n\to\R\).
  Let~\(\R\) denote the trivial representation of~\(\R^n\).  The
  continuous group cohomology for~\(\R^n\) with coefficients in a
  Fr\'echet \(\R^n\)-module~\(W\) is the same as
  \(\Ext^k_{L^1(\R^n)}(\R,W)\).  Since \(L^1(\R^n)\) is Abelian, the
  module structures on \(W\) and~\(\R\) induce an \(L^1(\R^n)\)-module
  structure on \(\Ext^k_{L^1(\R^n)}(\R,W)\) as well, and both module
  structures are the same.  Since one is trivial, so is the other.

  Let~\(V\) be a Fr\'echet space with a continuous action
  of~\(\torus^n\), which we view as an action of~\(\R^n\).  Split
  \(V\cong V^1\oplus V^2\) where \(V^1\subseteq V\) is the space of
  \(\torus^n\)-invariant elements and~\(V^2\) is the closed linear
  span of the other homogeneous components.  There is an element
  \(f\in L^1(\R^n)\) whose Fourier transform~\(\hat{f}\) satisfies
  \(\hat{f}(0)=1\) and \(\hat{f}(n)=0\) for all
  \(n\in\Z^n\setminus\{0\}\); for instance, we may take~\(f\) to be
  the inverse Fourier transform of a smooth bump function around~\(0\)
  supported in~\((-1,1)^n\).  The function~\(f\) acts on~\(V\) by the
  projection onto~\(V^1\).  The induced action by~\(f\) on the
  cohomology \(H^k(\R^n,V)\) is the same as the action of
  \(\hat{f}(0)\) because the latter is how~\(f\) acts on the trivial
  representation.  Hence \(H^k(\R^n,V^2)=0\), so
  \(H^k(\R^n,V)=H^k(\R^n,V^1)\).  In our case, the
  \(\torus^n\)-invariant elements in \(C(P,\R)\) are exactly the
  functions in \(C(P/\R^n,\R)\).

  This proves \(H^k(\R^n,C(P,\R))\cong H^k(\R^n,C(P/\R^n,\R))\),
  where the action of~\(\R^n\) on the Fr\'echet space
  \(C(P/\R^n,\R)\) is trivial.  Hence we get
  \(H^k(\R^n,C(P/\R^n,\R)) \cong H^k(\R^n,\R) \otimes_\R
  C(P/\R^n,\R)\), and \(H^k(\R^n,\R)\cong \extpd{k}{\R^n}\)
  as in the proof of \cite{Mathai-Rosenberg:T-duality}*{Lemma~2.1}.
\end{proof}

\begin{remark}
  \label{rem:toral_action}
  Proposition~\ref{pro:simplify_H3} still works if the stabiliser
  lattice~\(\Z^n\) of the action varies over~\(P\), that is, if there
  is a continuous function \(\gamma\colon P/\R^n\to
  \textup{Gl}_n(\R)/\textup{Gl}_n(\Z)\) such that, for each \(p\in
  P\), the lattice \(\gamma_{[p]}(\Z^n)\subseteq \R^n\) fixes~\(p\).
\end{remark}

Now we replace~\(\R^n\) by the crossed module~\(\cG\).  We assume that
the canonical map
\[
C(P/\R^n,\Omega^k\R)
\cong H^k_\mathrm{cont}(\R^n,C(P/\R^n,\R))
\to H^k_\mathrm{cont}(\R^n,C(P,\R))
\]
is an isomorphism for \(k=2,3\).  This holds, in particular, in the
situation of Proposition~\ref{pro:simplify_H3} or
Remark~\ref{rem:toral_action}.  Theorems
\ref{the:lift_to_cG}--\ref{the:Br_Pic_exact_sequence} extend to this
case, although we only state them in the situation of
Proposition~\ref{pro:simplify_H3}.

Under our assumption, the lifting obstruction in
\(H^3_\mathrm{cont}(\R^n,C(P,\R))\) is cohomologous to a \emph{unique}
function \(\psi\colon P/\R^n \to \extpd{3}{\R^n}\).  We also call this
function the \emph{lifting obstruction} of~\(A\).  The action
of~\(\R^n\) on~\(P\) lifts to an action on~\(A\) if and only if this
function \(P/\R^n \to \extpd{3}{\R^n}\) vanishes.

\begin{theorem}
  \label{the:lift_to_cG}
  Let~\(P\) be a second countable, locally compact \(\torus^n\)-space
  and let~\(A\) be a continuous trace \(\mathrm{C}^*\)-algebra with
  spectrum~\(P\).  Then the action of~\(\torus^n\) on~\(P\) lifts to
  an action of the bigroupoid \(\cG\ltimes P\) on~\(A\).
\end{theorem}

\begin{proof}
  Since \(\cG\ltimes P\) and \(\R^n\ltimes P\) have the same objects
  and arrows, we may construct \(E\) and~\(\omega\) as above.  For
  an action of~\(\cG\ltimes P\), we also need the datum~\(U\), and
  this modifies the cocycle condition for~\(\omega\).

  The condition~\ref{enum:U_vertical} for vertical products says
  that \(U(\xi,t,p)\) for fixed \(t\in\R^n\) and \(p\in P\) is a
  continuous homomorphism \(\extp{3}{\R^n}\to\torus\).  In
  particular, \(U(0,t,p)=1\) for all \(t,p\).  The
  condition~\ref{enum:U_horizontal} for horizontal products gives
  \[
  U(\xi_1,t_1,\alpha_{t_2}(p_2)) \cdot U(\xi_2,t_2,p_2)
  = U(\xi_1+\xi_2,t_1+t_2,p_2)
  \]
  for all \(\xi_1,\xi_2\in\extp{3}{\R^n}\), \(t_1,t_2\in\R^n\) and
  \(p\in P\).  For \(t_2=0\) and \(\xi_1=0\), this says that
  \(U(\xi_2,0,p_2) = U(\xi_2,t_1,p_2)\), so \(U(\xi,t,p)\) does not
  depend on~\(t\) and we may write \(U(\xi,t,p)\) as \(U(\xi,p)\); for
  \(t_1=0\) and \(\xi_2=0\),
  \ref{enum:U_horizontal} says \(U(\xi_1,\alpha_{t_2}(p_2)) = U(\xi_1,p_2)\),
  that is, \(U(\xi,p)\) only depends on the \(\R^n\)-orbit~\([p]\)
  of~\(p\).  For a function depending only on \(\xi\) and~\([p]\), the
  two multiplicativity conditions \ref{enum:U_vertical}
  and~\ref{enum:U_horizontal} are equivalent.  Thus
  \ref{enum:U_vertical} and~\ref{enum:U_horizontal} say exactly
  that~\(U\) is a continuous function from~\(P/\R^n\) to the space of
  group homomorphisms \(\extp{3}{\R^n}\to\torus\).  Such homomorphisms
  lift uniquely to homomorphisms \(\extp{3}{\R^n}\to\R\).  Thus
  \(U(\xi,t,p) = \exp(2\pi\ima \upsilon(\xi,[p]))\) for a continuous
  function \(\upsilon\colon P/\R^n \to \extpd{3}{\R^n}\) that is
  uniquely determined by~\(U\).

  Proposition~\ref{pro:simplify_H3} shows that there is a unique
  continuous function \(P/\R^n \to \extpd{3}{\R^n}\) that represents
  the lifting obstruction for an \(\R^n\)-action; that is, this
  function measures the failure of the cocycle
  condition~\ref{enum:omega_cocycle}, without the associator~\(U\),
  for a particular choice of~\(\omega\).  When we put in the
  associator, then~\ref{enum:omega_cocycle} holds if and only
  if~\(\upsilon\) above is equal to the lifting obstruction.  Hence there
  is an action \((E,\omega,U)\) of~\(\cG\ltimes P\) on~\(A\).
\end{proof}

Our next goal is a long exact sequence containing the Brauer and
Picard groups of~\(\cat\) and some known cohomology groups of \(P\)
and~\(P/\R^n\).  Let \(\Forg\colon \Br{\cG}{P}\to\Br{}{P}\) and
\(\Forg\colon \Pic{\cG}{P}\to\Pic{}{P}\) denote the forgetful maps.

\begin{theorem}
  \label{the:Br_Pic_exact_sequence}
  Let~\(P\) be a second countable, locally compact
  \(\torus^n\)-space.  There is a natural long exact sequence
  \begin{multline*}
    0 \leftarrow \Br{}{P}
    \xleftarrow{\Forg} \Br{\cG}{P}
    \leftarrow C(P/\R^n,\Omega^2 \R^n)
    \leftarrow \Pic{}{P}
    \xleftarrow{\Forg} \Pic{\cG}{P}
    \\\leftarrow C(P/\R^n,\R^n)
    \leftarrow H^1(P,\Z)
    \leftarrow H^1(P/\R^n,\Z)
    \leftarrow 0.
  \end{multline*}
  Here \(H^1({-},\Z)\) denotes \v{C}ech cohomology.  Furthermore,
  there are natural isomorphisms \(\Br{}{P}\cong H^3(P,\Z)\) and
  \(\Pic{}{P}\cong H^2(P,\Z)\).
\end{theorem}

\begin{proof}
  The surjectivity of \(\Forg\colon \Br{\cG}{P}\to\Br{}{P}\) is
  asserted in Theorem~\ref{the:lift_to_cG}.  Choose an action
  \((E,\omega, U)\) on~$A$ in~\(\Br{\cG}{P}\).  Any other action \((E',\omega',U')\) on the same~$A$ has
  \(E'\cong E\) by \ref{enum:E_unique}~in
  Lemma~\ref{lem:E_omega_U_exist}; this gives an equivalence to an
  action with \(E'=E\).  The proof of Theorem~\ref{the:lift_to_cG}
  shows that we cannot modify~\(U\) at all, that is, \(U=U'\), because
  it must be the exponential of the lifting obstruction of~\(A\).  The
  freedom in the choice of~\(\omega\) is to multiply it with
  \(\exp(2\pi\ima \varphi)\) for a continuous function \(\varphi\colon
  \R^n\times\R^n\times P\to\R\) normalised by
  \(\varphi(t,0,p)=\varphi(0,t,p)=0\); it must satisfy
  \(\partial\varphi=0\) so as not to violate~\ref{enum:omega_cocycle}.
  That is, \(\varphi\) is a cocycle for
  \(H^2_\mathrm{cont}(\R^n,C(P,\R))\).

  An equivalence of actions allows, among other things, to
  conjugate~\(\omega\) by an isomorphism \(V\colon E\to E\) that
  restricts to the identity on units.  By
  Lemma~\ref{lem:E_omega_U_exist}, this function~\(V\) differs from
  the identity by a continuous function \(\R^n\times P\to \torus\),
  which lifts uniquely to a continuous function \(\kappa\colon
  \R^n\times P\to \R\) normalised by \(\kappa(0,p)=0\) for all \(p\in
  P\).  Conjugating~\(\omega\) by the equivalence of actions defined
  by~\(\kappa\) multiplies it by \(\exp(2\pi\ima \partial\kappa)\)
  with
  \[
  \partial\kappa(t_1,t_2,p) = \kappa(t_2,p) - \kappa(t_1+t_2,p)
  + \kappa(t_1,\alpha_{t_2}(p)).
  \]
  Thus only the class of~\(\varphi\)
  in~\(H^2_\mathrm{cont}(\R^n,C(P,\R))\) matters for the equivalence
  class of the action.  Proposition~\ref{pro:simplify_H3} identifies
  \(H^2_\mathrm{cont}(\R^n,C(P,\R))\) with
  \(C(P/\R^n,\extpd{2}{\R^n})\); that is, any cocycle~\(\varphi\) is
  cohomologous to a unique one of the form \((t_1,t_2,p)\mapsto
  \chi[p](t_1\wedge t_2)\) for a continuous function \(\chi\colon
  P/\R^n \to \extpd{2}{\R^n}\).  Summing up, any action of~\(\cat\)
  on~\(A\) is equivalent to one having the same \(E\) and~\(U\) and
  \(\omega\cdot \exp(2\pi\ima\chi)\) for some \(\chi\in C(P/\R^n,
  \extpd{2}{\R^n})\).  Twisting the unit element of \(\Br{\cG}{P}\)
  with \(\exp(2\pi\ima\chi)\) as above defines a group homomorphism
  \(C(P/\R^n,\extpd{2}{\R^n})\to \Br{\cG}{P}\), which is one of the
  maps in our exact sequence.  Our argument shows that its range is
  the kernel of \(\Forg\colon \Br{\cG}{P}\to\Br{}{P}\), that is, our
  sequence is exact at \(\Br{\cG}{P}\).

  So far, we have only used equivalences of a special form.  In
  general, an equivalence between the two actions of~\(\cat\) on~\(A\)
  also involves a self-equivalence of~\(A\), that is, an element
  \(F\in\Pic{}{P}\).  This is given by a line bundle over~\(P\) by
  Lemma~\ref{lem:equivalence_continuous-trace}.  An equivalence
  \(A\sim A'\) allows to transport the given action \((E,\omega,U)\)
  on~\(A\) to an action \((E',\omega',U')\) on~\(A'\); now we apply
  this to the self-equivalence~\(F\).  This gives \(E' = r^*(L)
  \otimes_P E\otimes_P s^*(L)^*\), with \(\omega'\) and~\(U'\) defined
  by first cancelling pull-backs of \(L\otimes_P L^*\) in the middle
  and then applying \(\omega\) and~\(U\).  By
  Lemma~\ref{lem:E_omega_U_exist}~\ref{enum:E_unique}, there is an
  isomorphism \(E'\cong E\) that restricts to the identity on units.
  Using this, we transfer
  \(\omega'\) and~\(U'\) to~\(E\).  The argument above shows that,
  choosing the part~\(V\) in the equivalence suitably, we may arrange
  that \(\omega'=\exp(2\pi\ima\chi)\omega\) and \(U'=U\) for some
  \(\chi\in C(P/\R^n,\Omega^2\R^n)\); furthermore, \(\chi\) is
  independent of choices.

  Sending \(F\in\Pic{}{P}\) to this \(\chi\in C(P/\R^n,\Omega^2\R^n)\)
  gives a well-defined map \(\Pic{}{P}\to C(P/\R^n,\Omega^2\R^n)\);
  this is the next map in our exact sequence.  It is routine to check
  that this map is a group homomorphism.  Since we have now used the
  most general form of an equivalence, the actions \((E,\omega,U)\)
  and \((E,\exp(2\pi\ima\chi)\omega,U)\) are equivalent if and only
  if~\(\chi\) is in the image of \(\Pic{}{P} \to C(P/\R^n,\Omega^2\R^n)\);
  that is, our sequence
  is exact at \(C(P/\R^n,\Omega^2 \R^n)\).  Furthermore, if \(\chi=0\)
  then the equivalence~\((F,V)\) from \((E,\omega,U)\) to the twist
  by~\(\chi\) is a \emph{self}-equivalence, so we have lifted
  \(F\in\Pic{}{P}\) to \((F,V)\in\Pic{\cG}{P}\).  Thus our
  sequence is exact at \(\Pic{}{P}\).

  Now we consider the kernel of \(\Forg\colon
  \Pic{\cG}{P}\to\Pic{}{P}\); this consists of
  self-equivalences~\((F,V)\) of \((A,E,\omega,U)\) where~\(F\) is
  isomorphic to the identity equivalence on~\(A\); we may arrange
  \(F=A\) by a modification.  Any two choices for~\(V\) differ by
  pointwise multiplication with \(\exp(2\pi \ima \varphi)\) for some
  continuous function \(\varphi\colon \R^n\times P\to\R\) normalised
  by \(\varphi(0,p)=0\) for all \(p\in P\) by \ref{enum:V_unique} in
  Lemma~\ref{lem:E_omega_U_exist}.  Since~\(V\) already satisfies the
  coherence condition~\ref{enum:trafo_cocycle} for a cocycle,
  \(\varphi\) must be a cocycle for
  \(H^1_\mathrm{cont}(\R^n,C(P,\R))\).

  Among the modifications from \((F,V)\) to \((F,V')\) for
  some~\(V'\), we may consider those given by \(W=\exp(2\pi\ima
  \psi)\) for some \(\psi\colon P\to \R\).  This gives a modification
  from~\((F,V)\) to \((F,V\cdot\partial W)\) with \(\partial W(t,p) = W(\alpha_t
  p)W(p)^{-1}\).  Thus only the class of the
  cocycle~\(\psi\) in \(H^1_\mathrm{cont}(\R^n,C(P,\R))\) matters for
  the class in the Picard group.  Proposition~\ref{pro:simplify_H3}
  identifies \(H^1_\mathrm{cont}(\R^n,C(P,\R)) \cong C(P/\R^n,\R^n)\).
  So we get a surjective group homomorphism from \(C(P/\R^n,\R^n)\)
  onto the kernel of \(\Forg\colon \Pic{\cG}{P}\to\Pic{}{P}\).  This
  map continues our exact sequence and gives the exactness at
  \(\Pic{\cG}{P}\).

  So far, we have only used modifications that lift to a map
  \(P\to\R\); general modifications involve continuous maps
  \(\phi\colon P\to\torus=\R/\Z\).  Locally, such a map lifts
  to~\(\R\).  Choosing such a local lifting gives an open covering
  of~\(P\) and a subordinate \v{C}ech \(1\)-cocycle on~\(P\) with
  values in~\(\Z\).  There is a global lifting of~\(\phi\) to a
  function \(P\to\R\) if and only if this \(1\)-cocycle is a
  coboundary.  Any \v{C}ech \(1\)-cocycle in \(H^1(P,\Z)\) is the
  lifting obstruction of some continuous maps \(\phi\colon
  P\to\torus\) by a partition of unity argument.  Thus \(H^1(P,\Z)\)
  is the quotient of the group of all functions \(P\to\torus\) modulo
  those functions of the form \(\exp(2\pi\ima\psi)\) for a continuous
  function \(\psi\colon P\to\R\).

  Any function \(W\colon P\to\torus\) gives a modification between a
  given transformation~\((F,V)\) and another transformation
  \((F,V\cdot \partial W)\) with \(\partial W\) as above.
  Since this has the same underlying equivalence~\(F\), $\partial W$
  is of the form $\partial W = \exp(2\pi i \kappa)$ for a continuous
  map $\kappa \colon \R^n \times P \to \R$, which represents a
  $1$-cocycle in $H^1_{\mathrm{cont}}(\R^n, C(P,\R)) \cong
  C(P/\R^n, \R^n)$.  Thus there is a unique continuous function
  \(\chi\in C(P/\R^n,\R^n)\) and a function \(h\colon P\to\R\) such
  \(\partial W = \exp(2\pi\ima (\chi+\partial h))\).  Hence we get a
  well-defined
  map \(H^1(P,\Z)\to C(P/\R^n,\R^n)\) by sending the class of~\(W\) in
  \(H^1(P,\Z)\) to the unique~\(\chi\) above.  The image of this is
  the set of all~\(\chi\) for which \((F,V\exp(2\pi\ima\chi))\) is
  equivalent to~\((F,V)\).  This gives the exactness of our sequence
  at \(C(P/\R^n,\R^n)\).  Furthermore, \(\chi=0\) means that \(\partial
  (W/ \exp(2\pi\ima h))=1\) for some \(h\in C(P,\R)\).
  Equivalently, \(W/ \exp(2\pi\ima h))\) is an
  \(\R^n\)\nobreakdash-invariant function on~\(P\).
  This happens if and only if \([W]\in H^1(P,\Z)\) is in the image
  of \(H^1(P/\R^n,\Z)\);
  here the map \(H^1(P/\R^n,\Z)\to H^1(P,\Z)\) is induced by the
  quotient map \(P\to P/\R^n\).  We have shown exactness of our
  sequence at \(H^1(P,\Z)\).

  If \(\varphi\colon P/\R^n\to \torus\)
  goes to the trivial element of \(H^1(P,\Z)\), then
  \(\varphi = \exp(2\pi\ima\psi)\) for some \(\psi\colon P\to\R\).
  Since~\(\varphi\) is constant on \(\R^n\)-orbits and these are
  connected, \(\psi\) is also constant on \(\R^n\)-orbits, so we have
  lifted~\(\varphi\) to \(\psi\colon P/\R^n\to \R\).
  Thus~\(\varphi\) gives the trivial element of \(H^1(P/\R^n,\Z)\);
  that is, our sequence is exact also at \(H^1(P/\R^n,\Z)\).  The
  natural isomorphisms \(\Br{}{P}\cong H^3(P,\Z)\) and
  \(\Pic{}{P}\cong H^2(P,\Z)\) are well-known.
\end{proof}

\subsection{Making the action strict}
\label{sec:strictify}

So far, our actions on continuous trace \(\mathrm{C}^*\)-algebras are actions
by equivalences of the Lie bigroupoid~\(\cG\).  We are going to turn
these into strict actions of the crossed module~\(\cH\).  First, the
equivalence between \(\cH\) and~\(\cG\) shows that actions of
\(\cG\) and~\(\cH\) are ``equivalent.''  In particular,
\(\Br{\cG}{P} \cong \Br{\cH}{P}\).  Secondly, any action by
correspondences of the crossed module~\(\cH\) is equivalent to a
strict action by automorphisms, and equivalences among such actions
are equivariant Morita equivalences in an almost classical sense.

Let \((A,E,\omega,U)\) be an action of \(\cG\ltimes P\) by
correspondences.  This is a morphism from the bicategory
\(\cG\ltimes P\) to the correspondence bicategory, satisfying some
continuity conditions.  We may compose this with the morphism
\(\cat[\cH]\ltimes P\to\cG\ltimes P\) induced by \(\pi\colon
\cat[\cH]\to\cG\); this gives an action of \(\cat[\cH]\ltimes P\) by
correspondences.  Conversely, an action of \(\cat[\cH]\ltimes P\)
gives an action of \(\cG\ltimes P\) by composing with the morphism
induced by \(\iota\colon \cG\to\cat[\cH]\).  Going back and forth
induces a bijection between equivalence classes of actions because
\(\iota\) and~\(\pi\) are inverse to each other up to equivalence.
All this follows from general bicategory theory, which tells us how
to compose morphisms between bicategories, how to compose
transformations between such morphisms vertically and horizontally,
and that there are related multiplication operations on
modifications; in brief, bicategories with morphisms,
transformations and modifications form a tricategory.

Everything above also works for topological bicategories and
continuous morphisms, transformations and modifications.  The
correspondence bicategory, however, is not a topological bicategory
in the usual sense.  The continuity of a map from a locally compact
space~\(X\) to the set of \(\mathrm{C}^*\)-algebras is defined in an \emph{ad
  hoc} way by giving a \(C_0(X)\)-\(\mathrm{C}^*\)-algebra as an extra datum.
So continuity is not a property of a map, but extra structure.
Consider a continuous map \(f\colon Y\to X\) between locally compact
spaces and a continuous map from~\(X\) to \(\mathrm{C}^*\)-algebras, given by
a map \(x\mapsto A_x\) to \(\mathrm{C}^*\)-algebras and a
\(C_0(X)\)-\(\mathrm{C}^*\)-algebra~\(A\) with fibres~\(A_x\).  Then the
pull-back \(f^*(A) = C_0(Y) \otimes_{C_0(X)} A\) is a
\(C_0(Y)\)-\(\mathrm{C}^*\)-algebra with fibres~\(A_{f(y)}\); this is how we
compose a continuous map to \(\mathrm{C}^*\)-algebras with the continuous
map~\(f\).  Similarly, we may pull back \(\mathrm{C}^*\)-correspondences
and operators between them along continuous maps.  This is how we
compose continuous maps from locally compact spaces to the arrows
and bigons in the correspondence bicategory.  The general theory
implies, in particular:

\begin{theorem}
  \label{the:compare_Brauer}
  Let~\(P\) be a locally compact \(\R^n\)-space.  The morphisms
  \(\pi\) and~\(\iota\) between \(\cG\) and~\(\cH\) induce an
  isomorphism \(\Br{\cG}{P} \cong \Br{\cH}{P}\).  If~\(P\) is
  second countable and the \(\R^n\) action factors through~\(\torus^n\),
  then any continuous trace \(\mathrm{C}^*\)-algebra over~\(P\) carries an
  action of~\(\cH\) lifting the action of~\(\R^n\) on~\(P\).

  Let \((A',E',\omega',U')\) and \((A,E,\omega,U)\) be actions of
  \(\cH\) and~\(\cG\) corresponding to each other.  Then the
  equivariant Picard groups \(\Pic{\cG}{A,E,\omega,U}\) and
  \(\Pic{\cH}{A',E',\omega',U'}\) are canonically isomorphic.\qed
\end{theorem}

We also use the notation $\cH \ltimes P$ for the bigroupoid
$\cat[\cH] \ltimes P$.
Now we make explicit how a continuous action \((A,E,\omega,U)\)
of~\(\cG\ltimes P\) gives a continuous action \((A',E',\omega',U')\)
of~\(\cH\ltimes P\).  We put \(A'=A\), and \(E'_{(t,\eta)}=E_t\)
with \(E'= \pr_1^*(E)\), where \(\pr_1\colon
\R^n\times\extp{2}{\R^n}\to\R^n\) maps \((t,\eta)\mapsto t\).
Moreover, \(U'_{(\theta,\xi)} = U_\xi\) for all
\(\theta\in\extp{2}{\R^n}\), \(\xi\in\extp{3}{\R^n}\).  The
map~\(\omega'\) is defined by composing
\begin{multline*}
  E'_{(t_1,\eta_1,\alpha_{t_2}(p_2))} \otimes E'_{(t_2,\eta_2,p_2)}
  = E_{(t_1,\alpha_{t_2}(p_2))} \otimes E_{(t_2,p_2)}
  \xRightarrow{\omega(t_1,t_2,p_2)} E_{(t_1+t_2,p_2)}
  \\ \xRightarrow{U(\omega_\pi((t_1,\eta_1),(t_2,\eta_2)))}
  E_{(t_1+t_2,p_2)}
  = E'_{(t_1+t_2,\eta_1+\eta_2+t_1\wedge t_2,p_2)},
\end{multline*}
where \(\omega_\pi((t_1,\eta_1),(t_2,\eta_2)) = t_1\wedge \eta_2\)
and where the tensors are over \(A_{\alpha_{t_2}(p_2)}\).  Thus
\[
\omega'((t_1,\eta_1),(t_2,\eta_2,p_2))
= U(t_1\wedge\eta_2,p_2)\cdot \omega(t_1,t_2,p_2).
\]
It is routine to check that \((A,E',\omega',U')\) is a continuous
action of~\(\cH\ltimes P\).

Let \((A^i,E^i,\omega^i,U^i)\) for \(i=1,2\) be two continuous
actions of~\(\cG\ltimes P\) and let~\((F,V)\) be a transformation
between them.  Construct actions
\(((A^i)',(E^i)',(\omega^i)',(U^i)')\) of~\(\cH \ltimes P\) for \(i=1,2\) as
above.  The induced transformation~\((F',V')\) between these actions
of~\(\cH \ltimes P\) is given by \(F' = F\) and
\begin{multline*}
  V'_{(t,\eta,p)} = V_{t,p}\colon (E^1)'_{(t,\eta,p)} \otimes_{A^1_p} F_p
  = E^1_{(t,p)} \otimes_{A^1_p} F_p
  \\ \xRightarrow{V_{t,p}} F_{\alpha_t(p)} \otimes_{A^2_{\alpha_t(p)}} E^2_{(t,p)} =
  F_{\alpha_t(p)} \otimes_{A^2_{\alpha_t(p)}} (E^2)'_{(t,\eta,p)}.
\end{multline*}
This is indeed a transformation, and it remains an equivalence
if~\((F,V)\) is one.  Thus equivalent actions of~\(\cG\ltimes P\)
induce equivalent actions of~\(\cH\ltimes P\), as asserted by
Theorem~\ref{the:compare_Brauer}.  There is nothing to do to
transfer modifications between \(\cG\ltimes P\) and \(\cH\ltimes
P\).

Strict actions of crossed modules of topological groupoids are defined
in \cite{Buss-Meyer-Zhu:Non-Hausdorff_symmetries}.  We make this
explicit in the case we need:

\begin{definition}
  \label{def:strict_action_H}
  Let \(A\) be a \(C_0(P)\)-\(\mathrm{C}^*\)-algebra.  A strict action
  of~\(\cH\ltimes P\) on~\(A\) by automorphisms is given by
  continuous group homomorphisms \(\alpha\colon H^1\to \Aut{A}\) and
  \(u\colon H^2\to U(M(A))\) with \(\alpha_{\partial(h)} =
  \Ad{u(h)}\) for all \(h\in H^2\) and \(\alpha_{h}(u_k) =
  u_{c_h(k)}\) for all \(h\in H^1\), \(k\in H^2\), such that the
  homomorphism \(C_0(P)\to ZM(A)\) is \(H^1\)-equivariant,
  where~\(H^1\) acts on~\(P\) through the quotient map
  \(H^1\to\R^n\), \((t,\eta)\mapsto t\).
\end{definition}

Such a strict action of~\(\cH \ltimes P\) by automorphisms induces an
action by correspondences.  First, we take \(E=r^*(A)\) as a left Hilbert
\(A\)-module, with the right action of \(s^*(A)\) through~\(\alpha\):
\((x\cdot a)(t,\eta,p) = x(t,\eta,p)\cdot
\alpha_{t,\eta}(a(t,\eta,p))\) for all \(x\in E\), \(a\in s^*(A)\); so
\(a(t,\eta,p)\in A_p\) and \(\alpha_{t,\eta}(a(t,\eta,p))\in
A_{\alpha_t(p)} = A_{r(t,p)}\ni x(t,\eta,p)\), as it should be.  The
isomorphisms~\(\omega\) map \(x_1\otimes x_2\mapsto x_1\cdot
\alpha_{t,\eta}(x_2)\) for \(x_1\in E_{(t,\eta,p)}\), \(x_2\in
E_{(t_2,\eta_2,p_2)}\), and \(U_{((\theta,\xi),(t,\eta),p)}\colon
E_{(t,\eta,p)} \to E_{(t,\eta+\theta,p)}\) multiplies on the right
with the unitary~\(u(\theta,\xi)^*\).  This is an action by
correspondences as defined above.

\begin{theorem}
  \label{the:strictify_H-action}
  Any action of \(\cH\ltimes P\) by correspondences is equivalent to
  an action that comes from a strict action of \(\cH\ltimes P\) by
  automorphisms.
\end{theorem}

\begin{proof}
  This is contained in \cite{Buss-Meyer-Zhu:Higher_twisted}*{Theorem
    5.3}.
\end{proof}

By construction, the action by correspondences coming from a strict
action has \(E\cong r^*(A)\) as a left Hilbert \(A\)-module and
\(E\cong s^*(A)\) as a right Hilbert \(A\)-module.  Therefore, if
\(F\) is a \(\mathrm{C}^*\)-correspondence between \(A^1\) and~\(A^2\) for two
such actions \((A^1,E^1,\omega^1,U^1)\), then \(E^1\otimes_{s^*(A^1)}
s^*(F) \cong s^*(F)\) as a right Hilbert \(s^*(A)\)-module and
\(r^*(F)\otimes_{r^*(A^2)} E^2 \cong r^*(F)\) as a left Hilbert
\(r^*(A)\)-module.  In particular, these are isomorphisms of Banach
spaces.  The isomorphism~\(V\) in a transformation is therefore given
by bounded linear maps \(V_g\colon F_{s(g)}\to F_{r(g)}\); the extra
conditions to induce an isomorphism of correspondences
\(E^1\otimes_{s^*(A^1)} s^*(F) \to r^*(F)\otimes_{r^*(A^2)} E^2\) are
exactly the usual equivariance conditions for an equivariant
correspondence.  Thus the notion of equivalence for actions of
\(\cH\ltimes P\) by correspondences amounts to a standard notion of
equivariant Morita equivalence.

\begin{theorem}
  \label{the:strict_action_concrete}
  Let~\(A\) be a \(\mathrm{C}^*\)-algebra.  A strict action of~\(\cH\) on~\(A\)
  is equivalent to a pair of continuous maps \(\bar{\alpha}\colon
  \R^n\to\Aut{A}\) and \(\bar{u}\colon \extp{2}{\R^n}\to UM(A)\) with the
  following properties:
  \begin{enumerate}
  \item \(\bar{\alpha}_s\bar{\alpha}_t = \Ad{\bar{u}(s\wedge t)}
    \bar{\alpha}_{s+t}\) for all \(s,t\in\R^n\) and $\bar{\alpha}_0
    = \id{}$;
  \item \(\bar{u}\) is a group homomorphism;
  \item \(\bar{\alpha}_s(\bar{u}(t\wedge v))\bar{u}(t\wedge v)^* \in UM(A)\)
  is central and fixed by~\(\bar{\alpha}_w\) for all elements
  \(s,t,v,w\in \R^n\);
  \item \(\bar{\alpha}_s(\bar{u}(t\wedge v))\,\bar{u}(t \wedge v)^* =
          \bar{u}(s \wedge v)\,\bar{\alpha}_t(\bar{u}(s\wedge v))^*\)
		  for all \(s,t,v\in\R^n\).
  \end{enumerate}
\end{theorem}

\begin{proof}
  Let \(\bar{\alpha}\) and~\(\bar{u}\) have the required properties.  We claim
  that there are well-defined group homomorphisms
  \begin{alignat*}{2}
    \alpha \colon H^1 &\to \Aut{A},&\qquad
    (t,\eta)&\mapsto \Ad{\bar{u}(\eta)}\circ \bar{\alpha}_t,\\
    u \colon H^2 &\to UM(A),&\qquad
    (\theta,s\wedge \theta_2)&\mapsto
    \bar{u}(\theta)\cdot \bar{\alpha}_s(\bar{u}(\theta_2))\bar{u}(\theta_2)^*.
  \end{alignat*}
  The map~\(\alpha\) is clearly well-defined.  Since
  \(\bar{\alpha}_s(\bar{u}(\eta)) \bar{u}(\eta)^*\) is central, the automorphisms
  \(\Ad{\bar{u}(\eta)}\) and \(\Ad{\bar{\alpha}_s(\bar{u}(\eta))}\) are equal for all
  \(s\in\R^n\), \(\eta\in\extp{2}{\R^n}\).  Hence
  \begin{align*}
    &\alpha(t_1,\eta_1)\circ \alpha(t_2,\eta_2)
    = \Ad{\bar{u}(\eta_1)}\circ \bar{\alpha}_{t_1}\circ \Ad{\bar{u}(\eta_2)}
	\circ \bar{\alpha}_{t_2} \\
    =\ & \Ad{\bar{u}(\eta_1)}\circ \Ad{\bar{\alpha}_{t_1}(\bar{u}(\eta_2))}
	\circ \bar{\alpha}_{t_1}\circ \bar{\alpha}_{t_2}
    \\=\ &\Ad{\bar{u}(\eta_1)}\circ \Ad{\bar{u}(\eta_2)} \circ
    \Ad{\bar{u}(t_1\wedge t_2)} \circ \bar{\alpha}_{t_1 +t_2}
    = \Ad{\bar{u}(\eta_1) \bar{u}(\eta_2) \bar{u}(t_1\wedge t_2)} \circ
	\bar{\alpha}_{t_1 +t_2}
    \\=\ &\Ad{\bar{u}(\eta_1+\eta_2 +t_1\wedge t_2)} \circ \bar{\alpha}_{t_1 +t_2}
    = \alpha(t_1+t_2,\eta_1+\eta_2+ t_1\wedge t_2).
  \end{align*}
  Thus~\(\alpha\) is a homomorphism.  We have $u(\theta, \xi) =
  u(\theta,0)\,u(0,\xi)$ and $\theta \mapsto u(\theta,0)$ is a group
  homomorphism.  By assumption, the map \((s,t,v)\mapsto u(0,s\wedge t\wedge v)\) is
  antisymmetric and additive in \(t\) and~\(v\); thus it is additive also
  in~\(s\) and hence descends to a group homomorphism on \(\extp{3}{\R^n}\).
  Since \(u(0, s\wedge t\wedge v)\) is central, \(u\) is a well-defined group
  homomorphism.  We have \(\alpha_{(t,\eta)} (u_{\theta,\xi}) =
  u_{\theta,\xi+t\wedge \theta}\) and \(\alpha_{(0,\theta)} =
  \Ad{u(\theta,\xi)}\) by construction, so \(\alpha\) and~\(u\) form
  a strict action of the crossed module~\(\cH\).  Conversely, such a
  strict action \((\alpha,u)\) gives back \((\bar{\alpha},\bar{u})\) as above
  by taking \(\bar{\alpha}_t=\alpha_{t,0}\) and \(\bar{u}_\eta = u_{\eta,0}\).
\end{proof}

\begin{remark}
  \label{rem:crossed_module_crossed_product}
  Crossed products for crossed module actions are studied in
  \cite{Buss-Meyer-Zhu:Non-Hausdorff_symmetries} in the strict case,
  and in~\cite{Buss-Meyer:Crossed_products} for actions by
  correspondences.  This construction is, however, not very useful in
  our case because the crossed product is simply zero whenever the
  associator~\(U\) is non-trivial.  More precisely, the analysis
  in~\cite{Buss-Meyer:Crossed_products} shows that the crossed product
  factors through the quotient of~\(A\) by the ideal generated by
  \((U_\xi-1)a\) for all \(a\in A\), \(\xi\in\extp{3}{\R^n}\).  In the
  interesting case where the associator is needed, this quotient is
  zero.

  To get a non-zero \(\mathrm{C}^*\)\nobreakdash-algebra, we may first
  tensor~\(A\) by some other action~\(B\) of~\(\cH\) with the opposite
  associator, so that the diagonal action of~\(\cH\) on \(A\otimes B\)
  has trivial associator.  Thus \(A\otimes B\) carries a Green twisted
  action of~\(\R^n\), and the crossed product with~\(\cH\) on
  \(A\otimes B\) is the appropriate \(\R^n\)\nobreakdash-crossed
  product for such twisted actions.  A good choice for~\(B\) is the
  action of~\(\cH\) corresponding to the non-associative compact
  operators defined
  in~\cite{Bouwknegt-Hannabuss-Mathai:Nonassociative_tori}, see
  Theorem~\ref{the:non-assoc_compacts} for the corresponding Fell
  bundle over~\(\cG\).
\end{remark}

\subsection{Non-associative algebras}
\label{sec:non-assoc_Fell}

We now relate actions of the bigroup~\(\cG\) to
non-associative Fell bundles over the group~\(\R^n\) with a
trilinear associator.  We interpret these as continuous spectral
decompositions for non-associative \(\mathrm{C}^*\)-algebras as
in~\cite{Bouwknegt-Hannabuss-Mathai:Nonassociative_tori}.  In
particular, non-associative \(\mathrm{C}^*\)-algebras arise as
section \(\mathrm{C}^*\)-algebras for non-associative Fell bundles
over~\(\R^n\).

Let \((A,E,\omega,U)\) be an action of~\(\cG\) on a
\(\mathrm{C}^*\)-algebra.
This consists of imprimitivity \(A,A\)-bimodules~\(E_t\) for
\(t\in\R^n\) with a continuity structure \(E\subseteq
\prod_{t\in\R^n} E_t\), with a continuous
multiplication map \(\omega\colon \bigsqcup E_t \otimes_A E_u \to
\bigsqcup E_{t+u}\), and with a homomorphism \(\extp{3}{\R^n}\to
ZUM(A)\), \(\xi\mapsto U_\xi\), to the group \(ZUM(A)\) of central
unitary multipliers of~\(A\); here we use the same arguments as in the
proof of Theorem~\ref{the:lift_to_cG} to see that the isomorphism
\(E_t\to E_t\) associated to the bigon \(\xi\colon t\Rightarrow t\)
does not depend on~\(t\), belongs to the centre, and that the map
\(\xi\mapsto U_\xi\) is a group homomorphism.  In addition, these
arguments show that the unitaries~\(U_\xi\) are ``invariant'' under
the \(\R^n\)-action; this means here that left and right
multiplication by~\(U_\xi\) on~\(E_t\) gives the same map for
each~\(t\); this is more than being in the centre of~\(A\), it means
being in the centre of the whole Fell bundle.

We view the map~\(\omega\) as multiplication maps \(E_t\times E_u \to
E_{t+u}\).  The cocycle condition for~\(\omega\) then becomes
\begin{equation}
  \label{eq:associator_U}
  x_t \cdot (x_u \cdot x_v)
  = U_{t\wedge u\wedge v}\cdot ((x_t \cdot x_u) \cdot x_v)
  = ((x_t \cdot x_u) \cdot x_v)\cdot U_{t\wedge u\wedge v}
\end{equation}
for all \(x_t\in E_t\), \(x_u\in E_u\), \(x_v\in E_v\),
\(t,u,v\in\R^n\).  Thus our multiplication is not associative
unless~\(U\) is trivial, but in a controlled fashion given by the
associator~\(U\).  The multiplication determines the associator
uniquely by~\eqref{eq:associator_U}.  The existence of a continuous
homomorphism \(U\colon \extp{3}{\R^n}\to ZUM(A)\)
verifying~\eqref{eq:associator_U} is a restriction on the lack of
associativity of the multiplication map~\(\omega\).

There are unique conjugate-linear involutions \(E_t\to E_{-t}\),
\(x\mapsto x^*\), so that the left and right inner products on the
fibres~\(E_t\) are of the form \(x y^*\) and~\(x^* y\),
respectively.  We assumed that each~\(E_t\) is an imprimitivity
bimodule.  This is equivalent to the surjectivity of the
multiplication maps \(E_t \times E_u \to E_{t+u}\).  The two claims
above are proved in \cite{Buss-Meyer-Zhu:Higher_twisted}*{Theorem
  3.3} for actions of locally compact groups; the proof carries over
easily to the bigroup~\(\cG\) (see
also~\cite{Buss-Meyer:Crossed_products} for such results about
actions of crossed modules).  When we replace the inner products
on~\(E_t\) by the involutions as above, we get the data of a
\emph{saturated Fell bundle} over~\(\R^n\), except that our
multiplication is non-associative with a \(ZUM(A)\)-valued
associator.  Thus we may call this a \emph{non-associative saturated
  Fell bundle}.

When we interpret an associative, saturated Fell bundle
over~\(\R^n\) as a group action by correspondences, then the section
\(\mathrm{C}^*\)-algebra of the Fell bundle plays the role of the
crossed product with~\(\R^n\) for an action by automorphisms.  For a
non-associative Fell bundle~\(E\) as above, we may still define a
convolution and an involution on the space \(\Gamma_c(E)\) of
compactly supported, continuous sections of~\(E\) by
\[
(f_1 * f_2)(t) = \int_{\R^n} f_1(x) f_2(t-x) \,\mathrm{d}x
\]
and \(f^*(t) = f(-t)^*\) for \(t\in\R^n\).  This satisfies \((f^*)^* =
f\) and \((f_1 * f_2)^* = f_2^* * f_1^*\) as usual, but the
convolution is not associative.  Such non-associative
$^*$\nobreakdash-algebras and their
\(\mathrm{C}^*\)\nobreakdash-completions are studied in
\cites{Bouwknegt-Hannabuss-Mathai:Nonassociative_tori,
  Bouwknegt-Hannabuss-Mathai:Cstar_in_tensor}.  A non-associative
analogue of the bounded operators on Hilbert space is defined in
\cite{Bouwknegt-Hannabuss-Mathai:Cstar_in_tensor}*{Definition~4.2}.
Non-associative \(\mathrm{C}^*\)\nobreakdash-algebras are defined in
\cite{Bouwknegt-Hannabuss-Mathai:Cstar_in_tensor}*{Definition~B.4} as
norm-closed subalgebras of the non-associative bounded operators.
This implicitly defines the \(\mathrm{C}^*\)\nobreakdash-completion of
a non-associative $^*$\nobreakdash-algebra \(C_c(E)\), by considering
a supremum of norms over all \(^*\)\nobreakdash-homomorphisms to the
non-associative bounded operators.

The dual action of~\(\R^n\) on a crossed product by~\(\R^n\) still
works on~\(\Gamma_c(E)\): let \((\alpha_s f)(t) = \exp(2\pi\ima s t)
f(t)\) for \(s,t\in\R\) and \(f\in\Gamma_c(E)\).  This is a
$^*$\nobreakdash-automorphism of \(\Gamma_c(E)\), which extends to
the \(\mathrm{C}^*\)\nobreakdash-completion \(\mathrm{C}^*(E)\), and
the map \(s\mapsto \alpha_s(f)\) is continuous for each \(f\in
\mathrm{C}^*(E)\).

Conversely, consider a non-associative
\(\mathrm{C}^*\)\nobreakdash-algebra~\(B\) as
in~\cite{Bouwknegt-Hannabuss-Mathai:Nonassociative_tori}, equipped
with a continuous (dual) action~$\beta$ of~\(\R^n\).  A \emph{continuous
  spectral decomposition} of the \(\R^n\)\nobreakdash-action
on~\(B\) is, by definition, a continuous field of Banach
spaces~\(E_t\) over~\(\R^n\) with a continuous multiplication map
\(\bigsqcup E_t \times \bigsqcup E_s \to \bigsqcup E_{t+s}\) and a
continuous involution \(\bigsqcup E_t \to \bigsqcup E_{-t}\), such
that~\(B\) is \(\R^n\)\nobreakdash-equivariantly isomorphic to the
\(\mathrm{C}^*\)\nobreakdash-completion~\(\mathrm{C}^*(E)\) of the section
$^*$\nobreakdash-algebra~\(C_c(E)\); thus the multiplication
in~\(E\) has the same associator as~\(B\).  Elements of~\(E_t\) give
multipliers of~\(C_c(E)\) and hence of \(\mathrm{C}^*(E)\), and this
maps~\(E_t\) into the space
\[
M(B)_t = \{b\in M(B) \mid
\beta_s(b) = \exp(2\pi \ima s t)\cdot b\text{ for all }s\in\R^n\}
\]
of \(t\)-homogeneous multipliers of~\(B\).  Usually, however, the
spaces~\(M(B)_t\) are too large to give a continuous spectral
decomposition of~\(B\).  A continuous spectral decomposition need
not exist, and if it exists it is not unique.

For an associative \(\mathrm{C}^*\)-algebra~\(B\), this definition
of a continuous spectral decomposition goes back to
Exel~\cite{Exel:SpectralTheory}.  Continuous spectral decompositions
are related in~\cite{Buss-Meyer:Continuous} to Rieffel's theory of
generalised fixed point algebras.  This shows that continuous
spectral decompositions need not exist and need not be unique.  They
always exist for dual actions, however, because the Fell bundle
underlying the crossed product is a continuous spectral
decomposition.  Therefore, \(B\otimes \K(L^2\R^n) \cong (B\rtimes
\R^n)\rtimes \R^n\) always has a continuous spectral decomposition,
being a dual action.

By definition, a non-associative Fell bundle over~\(\R^n\) is the
same as a continuous spectral decomposition for an \(\R^n\)-action
on a non-associative \(\mathrm{C}^*\)-algebra with a unitary,
centre-valued associator.  Thus the non-associative
\(\mathrm{C}^*\)\nobreakdash-algebras
in~\cite{Bouwknegt-Hannabuss-Mathai:Nonassociative_tori} are very
close to non-associative Fell bundles over~\(\R^n\), and these are
the same as actions of the bigroup~\(\cG\).

As an example, we now construct a Fell bundle over~\(\cG\)
corresponding to the twisted compact operators $\K_\chi(L^2(\R^n))$
defined in
\cite{Bouwknegt-Hannabuss-Mathai:Nonassociative_tori}*{Section~5}.
Here $\chi \colon \extp{3}{\R^n} \to \textup{U}(1)$ is a fixed
tricharacter.

Let $A = C_0(\R^n)$. The bigroup~$\cG$ acts on~$A$ via correspondences
in the following way.  Let $E_t = C_0(\R^n)$ for \(t\in\R^n\), and let
\(A\) act on~$E_t$ on the left by a shifted pointwise multiplication,
\[
(a \cdot h)(s) = a(s + t)\,h(s)
\]
for $a \in A$ and $h \in E_t$, and on the right action by pointwise
multiplication, $(h \cdot a)(s) = h(s)a(s)$.  The inner product is the
poinwise one, \(\langle h_1, h_2\rangle(s) = \overline{h_1(s)}
h_2(s)\).  The space of \(C_0\)\nobreakdash-sections is $E = C_0(\R^n
\times \R^n)$ (we will use the first coordinate of~$E$ as the one
parametrizing the field over~$\R^n$).  The multiplication maps
\(\omega_{t_1,t_2} \colon E_{t_1} \otimes_A E_{t_2} \to E_{t_1 +
  t_2}\) on these fibres are
\[
\omega_{t_1,t_2}(h_1 \otimes h_2)(s) = \chi(t_1 \wedge t_2 \wedge s)\,
h_1(s + t_2)\,h_2(s),
\]
and $\xi \in \extp{3}{\R^n}$ acts by multiplication with the
scalar~$\chi(\xi)$ in each fibre~\(E_t\).  The triple $(A,E,\omega,U)$
satisfies \ref{enum:units_E}--\ref{enum:omega_cocycle} and therefore
defines a continuous action of~$\cG$ on~$A$ by correspondences.

Let $E^c = C_c(\R^n \times \R^n) \subseteq E$ be the subspace of
sections of compact support.  We define a (non-associative)
convolution on~\(E^c\) by
\begin{align*}		
  (f \bullet g)(t,s) & = \int_{\R^n} \omega_{r,t-r}(f(r,\cdot) \otimes g(t-r,\cdot))(s)\,\textup{d}r \\
  & = \int_{\R^n} \chi(r \wedge t \wedge s)\,f(r,s + t - r)\,g(t-r,s)\,\textup{d}r \\
  & = \int_{\R^n} \chi(r \wedge t \wedge s)\,f(r + t, s - r)\,g(-r,s)\,\textup{d}r
\end{align*}
Define $\Phi \colon C_c(\R^n \times \R^n) \to E^c$ by
$\Phi(K)(x,y) = K(x+y,y)$. Then
\begin{align*}
  (\Phi(K_1) \bullet \Phi(K_2))(t,s) &=
  \int_{\R^n} \chi(r \wedge t \wedge s)\, K_1(t + s, s - r)\,K_2(s-r,s)\,\textup{d}r\\
  & = \int_{\R^n} \chi((t + s) \wedge r \wedge s)\, K_1(t + s, r)\,K_2(r,s)\,\textup{d}r
\end{align*}
Thus, if $K_1 \cdot K_2$ denotes the non-associative convolution operation
on $C_c(\R^n \times \R^n)$ for the associator~$\chi$ as in
\cite{Bouwknegt-Hannabuss-Mathai:Nonassociative_tori}*{Section~5},
then $\Phi(K_1 \cdot K_2) = \Phi(K_1) \bullet \Phi(K_2)$.
	
The Fell bundle also comes with an involution given by the conjugate
linear bimodule homomorphism $E_t \to E_{-t}$ with $h^*(s) =
\overline{h(s + t)}$. For $f \in C_c(\R^n \times \R^n)$ we have
$f^*(t,s) = \overline{f(-t, s + t)}$.  The isomorphism~$\Phi$
intertwines the Fell bundle involution with the involution defined
in~\cite{Bouwknegt-Hannabuss-Mathai:Nonassociative_tori} because
\[
\Phi(K)^*(t,s) = \overline{\Phi(K)(-t,s+t)} = \overline{K(s,s+t)} =
\Phi(K^*)(t,s).
\]
	
The Hilbert--Schmidt norm on~$E^c$ is defined using the trace on
$A \cong E_0$ as follows:
\[
\lVert f \rVert_{HS}^2 = \int_{\R^n} (f^* \bullet f)(0,s)\,\textup{d}s
= \int_{\R^n \times \R^n} \lvert f(r, s) \rvert^2\,\textup{d}r\,\textup{d}s
\]
Thus~$\Phi$ preserves the norm as well. Therefore, it
extends to an isometric isomorphism of the non-associative
Hilbert--Schmidt operators.  Let $C^*(E)$ denote the
(non-associative) $C^*$\nobreakdash-algebra obtained by taking the closure of the
action of~$E^c$ on the space of Hilbert--Schmidt operators in the
associated operator norm.  The above observations imply:

\begin{theorem}
  \label{the:non-assoc_compacts}
  The non-associative $C^*$-algebra $C^*(E)$ associated to
  the Fell bundle over~\(\cG\) defined above is isomorphic to
  $\K_\chi(L^2(\R^n))$ via the unique continuous extension $\Phi
  \colon \K_\chi(L^2(\R^n)) \to C^*(E)$.
\end{theorem}

\section{Computing the lifting obstruction}
\label{sec:compute_obstruction}

Let~\(P\) be a \(\torus^n\)-space with orbit space~\(X\) and let~\(A\)
be a continuous trace \(\mathrm{C}^*\)-algebra over~\(P\).  We have seen that
the \(\torus^n\)-action on~\(P\) lifts to an \(\R^n\)-action on~\(A\) if
and only if a certain tricharacter \(\chi\colon \extp{3}{\R^n}\to
C(X,\R)\) vanishes.  How can we compute~\(\chi\)?

Let \(p\in P\), let \(1\le i<j<k\le n\), and let
\(\torus_{ijk}\subseteq \torus^n\) be the three-dimensional subtorus
given by the coordinates~\(i,j,k\).  The value of~\(\chi\) at
\((e_i\wedge e_j\wedge e_k,[p])\) may also be
computed by the smaller system where we replace \(\torus^n\), \(P\)
and~\(A\) by \(\torus_{ijk}\), \(\torus_{ijk}\cdot p\subseteq P\), and
the restriction of~\(A\) to the orbit~\(\torus_{ijk}\cdot p\); this is
because the lifting obstruction~\(\chi\) is natural.  Thus it suffices
to compute the lifting obstruction in the case of a transitive action
of a three-dimensional torus~\(\torus^3\) on a space~\(P\), with some
continuous trace \(\mathrm{C}^*\)-algebra~\(A\) over~\(P\).  If the stabilisers
in~\(P\) are not discrete, then~\(P\) is two-dimensional, so
\(H^3(P,\Z)=0\) and~\(A\) is a trivial bundle of compact operators.
In this case, the \(\torus^n\)-action clearly lifts to a
\(\torus^n\)-action on~\(A\), so the lifting obstruction vanishes.  So
we may assume that the stabilisers of points in~\(P\) are discrete.

We replace the action of~\(\torus^3\) by one of~\(\R^3\), which we
want to lift to~\(A\).  This action is still transitive, and by
assumption the stabiliser of a point is a lattice
\(\Gamma\subseteq\R^3\).  We choose coordinates in~\(\R^n\) so that
this lattice becomes \(\Z^3\subseteq\R^3\).  Thus we are reduced to
computing the lifting obstruction in the case where \(P=\R^3/\Z^3\)
with the standard \(\R^3\)-action.

Our theory tells us that the continuous trace \(\mathrm{C}^*\)-algebra carries
an action of the crossed module of Lie groupoids \(\cH\ltimes P\).
Since \(\R^3 \) acts transitively on~\(P\), this crossed module is
equivalent to the \(\cH\)-stabiliser of a point in~\(P\); for general
$n \in \N$ this is the crossed module of groups~\(\tilde{\cH}\) with
\(\tilde{H}^1 = \Z^n\times \extp{2}{\R^n}\) and \(\tilde{H}^2 =
\extp{2}{\R^n} \oplus \extp{3}{\R^n}\) with the restrictions of
\(\partial\) and~\(c\) from~\(\cH\).  In our situation, the usual
construction of induced actions of groups generalises to crossed
modules.  As we will see, any action of~\(\cH\ltimes P\) on our
continuous trace \(\mathrm{C}^*\)-algebra $A$ is induced from an action
of~\(\tilde{\cH}\), namely, the restriction of the action
of~\(\cH\ltimes P\) on a single fibre $A_p = \K$.

A strict action of $\tilde{\cH}$ on $\K$ consists of two group
homomorphisms \(\tilde{\alpha} \colon \tilde{H}^1 \to \Aut{\K}\)
and $\tilde{u} \colon \tilde{H}^2 \to U(M(\K))$ such that
\(\tilde{\alpha}_{\partial(h)} = \Ad{\tilde{u}(h)}\)
and
\(\tilde{u}(c_g(h)) = \tilde{\alpha}_g(\tilde{u}(h))\)
for $g \in \tilde{H^1}$, $h \in \tilde{H}^2$.  Let
\[
A = \mathrm{Ind}_{\tilde{H}^1}^{H^1}(\K) =
\{ f \in C_b(H^1, \K)\ |\ f(\tilde{g} \cdot g)
= \tilde{\alpha}_{\tilde{g}}(f(g)) \text{ for all }
\tilde{g} \in \tilde{H}^1,\ g \in H^1\}
\]
be the induced $\mathrm{C}^*$-algebra.  By
\cite{Raeburn-Williams:Morita_equivalence}*{Corollary~6.21} it has
continuous trace with spectrum $H^1/ \tilde{H}^1 = \torus^n$.  It
carries a canonical action $\alpha \colon H^1 \to \Aut{A}$ given by
$\alpha_h(f)(g) = f(g \cdot h)$.  Let $u \colon H^2 \to U(M(A))$ be
defined by \(u_k(g) = \tilde{u}(c_g(k))\) for $g \in H^1$ and $k \in
H^2$.  The pair $(\alpha,u)$ defines a strict action of $\cH \ltimes
P$ on~$A$ in the sense of Definition~\ref{def:strict_action_H}.  The
proof is straightforward, using that $\partial(H^2) \subset
\tilde{H}^1$ lies in the centre of~$H^1$.

Similarly, an equivalence $(F,V)$ between two strict actions
$(\tilde{\alpha},\tilde{u})$ and $(\tilde{\alpha}',\tilde{u}')$
induces an equivalence between the induced actions.  Since any action
of~$\tilde{\cH}$ on~$A$ by correspondences is equivalent to a strict
one by Theorem~\ref{the:strictify_H-action}, induction descends to a
well-defined group homomorphism \(\mathrm{Ind} \colon \Br{}{\tilde{\cH}}
\to \Br{\cH}{\torus^n}\).  Restricting an action to the fibre yields
another group homomorphism \(\mathrm{Res} \colon \Br{\cH}{\torus^n} \to
\Br{}{\tilde{\cH}}\).

\begin{lemma}
  \label{lem:ind_res}
  Let $\cH$ and~$\tilde{\cH}$ be the crossed modules described above
  acting on a stable continuous trace algebra~$A$ with spectrum
  $\torus^n$.  Restriction to the fibre of~$A$ and induction of a
  strict action to $\cH \ltimes \torus^n$ are inverse to each other
  and yield a group isomorphism \(\Br{}{\tilde{\cH}} \cong
  \Br{\cH}{\torus^n}\).
\end{lemma}

\begin{proof}
  If we restrict the induced action to the fibre over~$0$, we regain
  the action on~$\K$ we started with.  Thus \(\mathrm{Res} \circ
  \mathrm{Ind} = \id{\Br{}{\tilde{\cH}}}\).

  Let $p_0 = 0 \in \torus^n = \R^n/\Z^n$.  Let $(\alpha,u)$ be a
  strict action of $\cH \ltimes \torus^n$ on~$A$.  Let
  $(\tilde{\alpha}, \tilde{u})$ be the restricted action of
  $\tilde{H}^1$ on $A(p_0)$, which we identify with~$\K$.  Let \(A' =
  \mathrm{Ind}_{\tilde{H}^1}^{H^1}(\K)\) and denote the induced action by
  $(\alpha', u')$.  Consider the $C(\torus^n)$-algebra homomorphism
  $\varphi \colon A \to A'$, which maps $a \in A$ to $f_a$ with
  $f_a(g) = \alpha_{g}(a)(p_0)$.  Since $f_{\alpha_h(a)}(g) =
  f_a(gh)$, it is equivariant.  It is norm-preserving and therefore
  injective.  Local triviality and a partition of unity argument show
  that it is also surjective, hence an isomorphism.  The extension
  of~$\varphi$ to the multiplier algebra maps~$u$ to~$u'$.  Therefore,
  $\varphi$ is an isomorphism that intertwines the two actions of $\cH
  \ltimes \torus^n$, which implies $\mathrm{Ind} \circ \mathrm{Res} =
  \id{\Br{\cH}{\torus^n}}$.
\end{proof}

To compute $\Br{}{\tilde{\cH}}$ we may further simplify the situation
by weakening.  The equivalence between \(\cH\) and~\(\cG\) restricts
to one between \(\tilde\cH\) and~\(\tilde\cG\), where
\(\tilde{G}^1=\Z^n\), \(\tilde{G}^2=\extp{3}{\R^n}\) with the restriction
of the bigroupoid structure from~\(\cG\).  Since~\(\K\) is equivalent
to~\(\C\), strict actions of~\(\tilde\cH\) on~\(\K\) by
automorphisms are equivalent to actions of~\(\tilde\cG\) by
correspondences on~\(\C\).

Such an action has \(A=\C\) and \(E_g=\C\) for all
\(g\in\tilde{G}^1=\Z^n\) because this is, up to isomorphism, the
only imprimitivity bimodule from~\(\C\) to itself.  The
multiplication maps and the action of bigons give maps
\(\omega\colon \Z^n\times\Z^n \to \torus\) and \(U\colon
\extp{3}{\R^n}\to\torus\) because an isomorphism \(\C\to\C\) is
simply multiplication by a scalar of modulus one.  The conditions
for an action require the following:
\begin{enumerate}
\item \(\omega(0,l)=\omega(k,0)=1\) for all \(k,l\in\Z^n\);
\item \(U\) is a continuous group homomorphism;
\item \(\omega(k+l,m)\omega(k,l) U(k\wedge l\wedge m) =
  \omega(k,l+m)\omega(l,m)\) for all \(k,l,m\in\Z^n\).
\end{enumerate}

For a discrete group $\Gamma$ and a $\Gamma$-module $M$, we denote
the group cohomology of $\Gamma$ with coefficients in $M$ by
$H^k_\mathrm{gr}(\Gamma,M)$.

\begin{theorem}
  \label{thm:compute_torus}
  Let $k = \binom{n}{3}$.  The pair $(\omega, U)$ associated to an
  action of $\tilde{\cG}$ on $\C$ has the following properties:
  \begin{enumerate}
  \item $U$ is trivial on $\extp{3}{\Z^n}$ and therefore
    yields a character on $\torus^{k} =
    \extp{3}{\R^n} / \extp{3}{\Z^n}$,
  \item $\omega$ is a group $2$-cocycle representing a central
    extension of $\Z^n$ by $\torus$.
  \end{enumerate}
  Moreover, we have group isomorphisms
  \[
  \Br{\cH}{\torus^n} \cong \Br{}{\tilde{\cH}} \cong
  \Br{}{\tilde{\cG}} \cong \Z^{k} \times H^2_\mathrm{gr}(\Z^n,\torus).
  \]
\end{theorem}

\begin{proof}
  It follows from (c) in the conditions listed above that the
  restriction of~$U$ to $\extp{3}{\Z^n}$ is the coboundary
  of~$\omega$.  In particular, it represents the trivial element in
  cohomology.  This is impossible if~$U$ is a non-trivial tricharacter
  on~$\Z^n$.  Thus~$\omega$ is a $2$-cocycle classifying a central
  $\torus$-extension of~$\Z^n$ and~$U$ is a character on
  $\extp{3}{\R^n}/\extp{3}{\Z^n} \cong \torus^k$.  The group of all
  such characters is the Pontrjagin dual of~$\torus^k$, which
  is~$\Z^k$.

  An equivalence $(F,V)$ between two actions $(\omega, U)$ and
  $(\omega',U')$ has to have $F = \C$, and $V \colon \Z^n \to \torus$
  has to satisfy $V(0) = 1$ and
  \(\omega(k,l) V(k+l) = V(l) V(k)\,\omega'(k,l)\)
  for all $k,l \in \Z^n$.  Therefore $(\omega, U)$ and $(\omega',U')$ are
  equivalent if and only if $U = U'$ and $\omega$ differs from
  $\omega'$ by a coboundary.  Moreover, the product of elements in the
  Brauer group of the bigroupoid $\tilde{\cG}$ translates into
  the product of characters and cocycles.  Since conditions (a)--(c)
  above fully characterise an action of $\tilde{\cG}$ on $\C$ by
  correspondences, every pair $(\omega,U)$ is associated to such
  an action.  Altogether, we have constructed a group isomorphism
  \(\Br{}{\tilde{\cG}} \to \Z^k \times H^2_\mathrm{gr}(\Z^n,\torus)\).
  The isomorphism $\Br{\cH}{\torus^n} \cong \Br{}{\tilde{\cH}}$ was
  constructed in Lemma~\ref{lem:ind_res}, and
  \(\Br{}{\tilde{\cG}} \cong \Br{}{\tilde{\cH}}\)
  follows from the equivalence of the two bigroupoids.
\end{proof}

We can make the inverse of the isomorphism \(\Br{\cH}{\torus^n} \cong
\Z \times H^2_\mathrm{gr}(\Z^n, \torus)\) more explicit: We first have to
transfer the action of~\(\tilde\cG\) to~\(\tilde\cH\) using the
functor~\(\pi\), then make that weak action strict by passing to a
stabilisation by \cite{Buss-Meyer-Zhu:Higher_twisted}*{Theorem 5.3}.
We have \(U_{(\theta,\xi)} = U^{\tilde{\cG}}_{\xi}\) for
\(\theta\in\extp{2}{\R^n}\), \(\xi\in\extp{3}{\R^n}\) because
\(\pi(\theta,\xi)=\xi\).  For \(k_1,k_2\in\Z^n\),
\(\eta_1,\eta_2\in\extp{2}{\R^n}\), the natural transformation
\(\pi(k_1,\eta_1)\cdot \pi(k_2,\eta_2) \Rightarrow
\pi((k_1,\eta_1)\cdot (k_2,\eta_2))\) in~\(\tilde\cG\) is \(k_1\wedge
\eta_2\), so \(\omega((k_1,\eta_1),(k_2,\eta_2)) =
U^{\tilde{\cG}}_{k_1\wedge \eta_2}\cdot\omega^{\tilde{\cG}}(k_1,
k_2)\).

To turn this action of the crossed module~\(\tilde\cH\) by
correspondences into a strict action by automorphisms, we take the
Hilbert space of \(L^2\)-sections of the Fell bundle~\((\C)_{g\in
  \tilde{H}^1}\) over~\(\tilde{H}^1\).  In our case, this is simply
\(\mathcal{K} = L^2(\Z^n\times\extp{2}{\R^n})\).  The multiplication
in the Fell bundle given by~\(\omega\) forms an action
of~\(\tilde\cH\) on this Hilbert module over~\(\C\).

We give \(\K(\mathcal{K})\) the induced action, so that~\(\mathcal{K}\) is
an equivariant Morita equivalance between \(\K(\mathcal{K})\)
and~\(\C\) with the given action of~\(\tilde\cH\).  Since the action
of~\(\tilde\cH\) on~\(\C\) is non-trivial, \(\mathcal{K}\) is not a
Hilbert space representation of~\(\tilde\cH\), but only a projective
Hilbert space representation.  On~$\tilde{H}^1$, this is given by
\[
((k_2, \eta_2) \cdot f)(k_1, \eta_1) =
\omega((k_1, \eta_1), (k_2, \eta_2))f(k_1 + k_2,
\eta_1 + \eta_2 + k_1 \wedge k_2)
\]
for $k_1,k_2 \in \Z^n$ and $\eta_1, \eta_2 \in \extp{2}{\R^n}$.  This
induces an action $\alpha \colon \tilde{H}^1 \to
\Aut{\K(\mathcal{K})}$.  Together with the homomorphism $u \colon
\tilde{H}^2 \to U(\mathcal{K})$ given by
\[
u_{(\theta, \xi)}(f)(k, \eta) =
U^{\tilde{\cG}}(\xi + k \wedge \theta)\,f(k, \eta + \theta)
\]
this is the strict action of~$\tilde{\cH}$ on~$\K(\mathcal{K})$ that
corresponds to the action by correspondences of~$\tilde{\cG}$ on~$\C$
given by $(\omega^{\tilde{\cG}}, U^{\tilde{\cG}})$.

Finally, we induce the action of~\(\tilde\cH\)
on~\(\K(\mathcal{K})\) to~\(\cH \ltimes \torus^n\).  As described above,
this produces an action of \(\cH\ltimes P\) with \(P=\R^n/\Z^n\).  More
precisely, we get an action of~\(\cH\ltimes P\) on a continuous trace
\(\mathrm{C}^*\)-algebra over~\(P\).

\begin{theorem}
  \label{the:generator}
  Let $A$ be the continuous trace $\mathrm{C}^*$-algebra that corresponds to
  the element $(1,0) \in \Z \times H^2_\mathrm{gr}(\Z^3, \torus) \cong
  \Br{\cH}{\torus^3}$.  Then the Dixmier--Douady invariant of~$A$ is
  a generator of $H^3(\torus^3,\Z)$.
\end{theorem}

\begin{proof}
  Let $(0,x) \in \Z \times H^2_\mathrm{gr}(\Z^3, \torus)$ and choose a
  cocycle $\omega \colon \Z^3 \times \Z^3 \to \torus$
  representing~$x$.  The action of~$\tilde{\cH}$ on~$\C$ by
  correspondences associated to this pair is pulled back from an
  action of~$\Z^3$ on~$\C$ by correspondences via the canonical
  functor $\tilde{\cH} \to \Z^3$.  Let $\mathcal{K}' = L^2(\Z^3)$.
  The group~$\Z^3$ acts projectively on $\mathcal{K}'$ via~$\omega$.
  This induces an honest representation $\alpha \colon H^1 \to \Z^3
  \to \Aut{\K(\mathcal{K}')}$ of~$H^1$ on the compact operators.  Let
  $u \colon H^2 \to U(M(\K(\mathcal{K}')))$ be the trivial
  homomorphism.  The pair $(\alpha, u)$ is a strict action
  of~$\tilde{\cH}$ on $\K(\mathcal{K}')$.  By the same reasoning as
  above, we may choose~$A$ to be the $\mathrm{C}^*$-algebra obtained by
  inducing this $\tilde{\cH}$-action to an $\cH$-action.  Since $A
  \cong C(\torus^3, \K(\mathcal{K}'))$, its Dixmier--Douady class
  vanishes.  But \(\Br{\cH}{\torus^3} \to \Br{}{\torus^3} \cong
  H^3(\torus^3, \Z)\) is surjective, therefore $(1,0)$ has to be
  mapped to a generator of $H^3(\torus^3, \Z) \cong \Z$.
\end{proof}

\begin{corollary}
  \label{cor:n-torus}
  Let $k = \binom{n}{3}$.  The group homomorphism
  \[
  \Z^k \times H^2_\mathrm{gr}(\Z^n, \torus) \cong
  \Br{\cH}{\torus^n} \to H^3(\torus^n, \Z),
  \]
  which maps an element $([U], [\omega])$ to the Dixmier--Douady class
  of the associated continuous trace $\mathrm{C}^*$-algebra restricts to an
  isomorphism $\Z^k \to H^3(\torus^n,\Z)$ and maps $H^2_\mathrm{gr}(\Z^n,
  \torus)$ to zero.
\end{corollary}

\begin{proof}
  Each projection $p_{ijk} \colon \torus^n \to \torus^3$ for $1 \leq i
  < j < k \leq n$ induces a commutative diagram
  \[
  \xymatrix{
    \Z^k \times H^2_\mathrm{gr}(\Z^n, \torus) \ar[r] &
    H^3(\torus^n, \Z) \\
    \Z \times H^2_\mathrm{gr}(\Z^3, \torus) \ar[u]^-{p_{ijk}^*}
    \ar[r] &
    H^3(\torus^3, \Z). \ar[u]_-{p_{ijk}^*}
  }
  \]
  Since the vertical arrows are split by corresponding inclusions
  $\torus^3 \to \torus^n$, they are injective.  In particular,
  $p_{ijk}^*(1) \in \{\pm e_i \wedge e_j \wedge e_k\} \subset
  \extp{3}{\Z^n} \cong \Z^k$.  The statement now follows from
  Theorem~\ref{the:generator} because $\Z^k \times H^2_\mathrm{gr}(\Z^n,
  \torus)$ is generated by the images of all~$p_{ijk}^*$.
\end{proof}

\section{Crossed module actions and T-duality}

Let $A$ be a continuous trace $\mathrm{C}^*$-algebra whose spectrum is a
$\torus^n$-space~$P$ with orbit space~$X$.  A second such pair~$A'$
with spectrum~$P'$ and the same orbit space is said to be
(topologically) \emph{T-dual} to~$A$ if the $\R^n$-action on~$P$ lifts
to one on~$A$ with \(A' \cong A \rtimes \R^n\).  In particular, this
implies an isomorphism of twisted $K$-groups $K_*(A) \cong K_{* -
  n}(A')$ by the Connes--Thom isomorphism.  In the case of principal
$\torus^1$-bundles, any pair $(A,P)$ has a unique T-dual.

A T-dual need no longer exist for higher-dimensional torus bundles.
The first obstruction against it is the lifting obstruction discussed
above.  Even if it vanishes, the crossed product need not be a
continuous trace $\mathrm{C}^*$-algebra.  Whether this is the case is
determined by a class in $H^1(X, \Z^\ell)$ for $\ell = \binom{n}{2}$,
which is derived from the Mackey obstruction as in
\cite{Mathai-Rosenberg:T-duality_II}*{Theorem~3.1}.  If it does not
vanish, the crossed product turns out to be a bundle of
non-commutative tori.

A (non-associative) Fell bundle $(E,\omega,U)$ given by an action of
$\cH \ltimes P$ on $A$ by correspondences combines all of the
obstructions into one structure: We have already identified the
lifting obstruction.  By Theorem~\ref{thm:compute_torus}, the Brauer
group of the fibre \(\Br{\cH}{\torus^n} \cong H^2_\mathrm{gr}(\Z^n,
\torus) \times H^3(\torus^n,\Z)\) may be interpreted as the group of
all possible Dixmier--Douady classes and Mackey obstructions
of~$\torus^n$.  As described in \cite{Olesen-Raeburn:pointw-unitary},
the group $H^2_\mathrm{gr}(\Z^n, \torus)$ can be equipped with a natural
topology, and \cite{Olesen-Raeburn:pointw-unitary}*{Lemma~3.3} gives a
homomorphism
\[
M \colon \Br{\cH}{P} \to C(X, H^2_\mathrm{gr}(\Z^n, \torus)),
\]
which sends $[A] \in \Br{\cH}{P}$ to the function that maps~$x$
to the Mackey obstruction of $[A(x)] \in \Br{\cH}{\torus^n}$.
The homotopy class
\(
[M(A)] \in \pi_0(C(X, H^2_\mathrm{gr}(\Z^n,\torus))) \cong
H^1(X, \Z^{\ell})
\)
for $\ell = \binom{n}{2}$ vanishes if and only if there is a
classical T-dual, see
\cite{Mathai-Rosenberg:T-duality_II}*{Theorem~3.1}.

To summarise: The lifting obstruction of the Fell bundle
$(E,\omega,U)$ vanishes if and only if the multiplication in
the Fell bundle is associative.  If it is, then we can form
the section $\mathrm{C}^*$-algebra associated to the Fell bundle.
This will again be of continuous trace if and only if $[M(A)]$
vanishes.  If it does, then the section algebra is the classical
T-dual of $A$, else it is a non-commutative T-dual.  As can
be seen from this, the non-associative Fell bundle obtained from
the crossed module action contains all information of the
T-dual in case it exists and all residual information in case
it does not.

\end{document}